\documentclass[12pt]{amsart} % for proof reading

\allowdisplaybreaks % for proof reading
\usepackage{amssymb,amsmath}

\theoremstyle{plain}
\newtheorem{theorem}{Theorem}[section]
\newtheorem*{theorem*}{Theorem}
\newtheorem{lemma}[theorem]{Lemma}

\newtheorem{proposition}[theorem]{Proposition}
\newtheorem{definition}[theorem]{Definition}
\theoremstyle{remark}
\newtheorem{remark}[theorem]{Remark}
\numberwithin{equation}{section}

\newcount\quantno
\everydisplay{\quantno=0}\everycr{\quantno=0}
\newcommand\quant{\advance\quantno by1
                      \ifnum\quantno=1\qquad\else\quad\fi\forall }
\newcommand\itemno[1]{(\romannumeral #1)}

\renewcommand\Re{\operatorname{\mathrm{Re}}}
\renewcommand\Im{\operatorname{\mathrm{Im}}}

\newcommand\rest[1]{\kern-.1em
          \lower.5ex\hbox{$\scriptstyle #1$}\kern.05em}

\renewcommand\mod[1]{\left\vert{#1}\right\vert}
\newcommand\bigmod[1]{\bigl\vert{#1}\bigr|}
\newcommand\Bigmod[1]{\Bigl\vert{#1}\Bigr|}

\newcommand\norm[2]{{\Vert{#1}\Vert_{#2}}}

\newcommand\bignorm[2]{\left.{\bigl\Vert{#1}\bigr\Vert_{#2}}\right.}
\newcommand\Bignorm[2]{\left.{\Bigl\Vert{#1}\Bigr\Vert_{#2}}\right.}

\newcommand\opnorm[2]{|\!|\!| {#1} |\!|\!|_{#2}}
\newcommand\prodo[2]{\left\langle#1,#2\right\rangle}

\newcommand\smallfrac[2]{\mbox{\small$\displaystyle\frac{#1}{#2}$}}
\newcommand\wrt{\,\text{\rm d}}

\newcommand\bB{\mathbf{B}}

\newcommand\bE{\mathbf{E}}

\newcommand\bP{\mathbf{P}}

 \newcommand\bs{\mathbf{s}}

\newcommand\bV{\mathbf{V}}
\newcommand\bW{\mathbf{W}}

\newcommand\BC{\mathbb{C}}

\newcommand\BN{\mathbb{N}}
\newcommand\BR{\mathbb{R}}

  \newcommand\fra{\mathfrak{a}}
\newcommand\cB{\mathcal{B}}

 \newcommand\frg{\mathfrak{g}}
\newcommand\cH{\mathcal{H}} \newcommand\fH{\mathfrak{H}}

  \newcommand\frk{\mathfrak{k}}
\newcommand\cL{\mathcal{L}}  
\newcommand\cM{\mathcal{M}}  
\newcommand\cN{\mathcal{N}}  \newcommand\frn{\mathfrak{n}}
  
  \newcommand\frp{\mathfrak{p}}

\newcommand\Rplus{{\mathbb R}^+}

\newcommand\al{\alpha}
\newcommand\be{\beta}
\newcommand\ga{\gamma} \newcommand\Ga{\Gamma}
\newcommand\de{\delta}
 \newcommand\vep{\varepsilon}

\newcommand\la{\lambda} \newcommand\La{\Lambda}
\newcommand\om{\omega} \newcommand\Om{\Omega} 
\newcommand\si{\sigma}

\newcommand\vp{\varphi}

\newcommand\OV{\overline}
\newcommand\funnyk{k\hbox to 0pt{\hss\phantom{g}}}

\newcommand\lu[1]{L^1(#1)}

\newcommand\laq[1]{L^q(#1)}
\newcommand\ld[1]{L^2(#1)}

\newcommand\ly[1]{L^\infty(#1)}
\newcommand\lorentz[3]{L^{#1,#2}(#3)}

\newcommand\bc{\mathbf{c}}
\newcommand\vbc{\rlap{\v}{{\textbf{c}}}}

\newcommand\wt{\widetilde}
\newcommand\whH{\widehat{\phantom{G}}\hbox to 0pt{\hss $H$}}

\newcommand\emspace{\hbox to 6pt{\hss}}

\newcommand\rmi{\hbox{\rm (i)}}
\newcommand\rmii{\hbox{\rm (ii)}}

\newcommand\rmiv{\hbox{\rm (iv)}}

\newcommand\rmvii{\hbox{\rm (vii)}}

\newcommand\rmix{\hbox{\rm (ix)}}

\newcommand\One{{\mathbf{1}}}

\newcommand\e{\mathrm{e}}

\newcommand\sft[1]{\wt{#1}}

\newcommand\planc{\mod{\mathbf{c}(\la)}^{-2} \wrt \la}
\newcommand\planl{\wrt \mu(\la)}
\newcommand\muD{\mu^s}
\newcommand\KX{K \backslash X}

\newcommand\KGK{{K \backslash G/K}}
\newcommand\alza[1]{\kern+.13em 
          \raise.35ex\hbox{$\scriptstyle #1$}}
\newcommand\alzaprima[1]{\kern-.25em
          \raise.2ex\hbox{$\scriptstyle #1$}}
\newcommand\GOp[1]{{\phantom{,}^{\alzaprima{G}}\cB^{\alza#1}(X)}}
\newcommand\Cone[1]{{\mathbf{\Gamma}_{#1}}}
\newcommand\Ball[1]{{\mathbf{B}_{#1}}}
\newcommand\NBall[1]{{\mathbf{b}_{#1}}}
\newcommand\Parabolic{{\mathbf{p}}}
\newcommand\Strip{{\mathbf{s}}}
\newcommand\OTW{\OV T_{\mathbf{W}}}
\newcommand\TW{T_{\mathbf{W}}}
\newcommand\TWpiu{T_{\mathbf{W}^+}}
\newcommand\OTWpiu{\OV{T}_{\mathbf{W}^+}}
\newcommand\TWpiut{T_{\mathbf{W}^t}}

\newcommand\TB{T_{\mathbf{B}}}
\newcommand\TE{T_{\mathbf{E}}}
\newcommand\TEpiu{T_{\mathbf{E}^+}}
\newcommand\TBpiu{T_{\mathbf{B}^+}}
\newcommand\TBpiut{T_{\mathbf{B}^t}}

\newcommand\lae{\lambda_{\eta}}

\newcommand\Cosh{\mathrm{Cosh}}

\newcommand\kBdue{k_{B_2}}
\newcommand\kBdueO{k_{B_2}^{(0)}}
\newcommand\kBdueI{k_{B_2}^{(1)}}

\newcommand\kMdue{k_{M_2(\cL)}}
\newcommand\kMdueO{k_{M_2(\cL)}^{(0)}}
\newcommand\kMdueI{k_{M_2(\cL)}^{(1)}}

\newcommand\astarc{\mathfrak{a}_{\mathbb{C}}^*}

\newcommand\aFb{\mathfrak{a}_F}
\newcommand\aFa{\mathfrak{a}^F}
\newcommand\astFa{(\mathfrak{a}^*)^F}
\newcommand\astFb{(\mathfrak{a}^*)_F}
\newcommand\nFb{\mathfrak{n}_F}
\newcommand\nFa{\mathfrak{n}^F}
\newcommand\rFb{\rho_F}
\newcommand\rFa{\rho^F}
\newcommand\QFb{\La_F}
\newcommand\QFa{\La^F}
\newcommand\HFb{H_F}
\newcommand\HFa{H^F}
\newcommand\lFb{\ell_F}
\newcommand\lFa{\ell^F}
\newcommand\Sp{\Sigma^+}
\newcommand\Spp{\Sigma_0^{+}}
\newcommand\Sppp{\Sigma_s}
\newcommand\SpFb{(\Sigma_F)^+}

\newcommand\SppFb{(\Sigma_F)_0^{+}}

\newcommand\bcFb{\mathbf{c}_F}
\newcommand\bcFa{\mathbf{c}^F}

\newcommand\height[1]{|#1|}
\newcommand\omFa{\omega^{F}}
\newcommand\omFaast{\omega^{*F}}

\newcommand\omFbast{\omega_F^*}
\newcommand\WFb{W_F}

\newcommand\phiFa[2]{\vp^F_{#1,#2}}
\newcommand\MFb{M_F}
\newcommand\MFa{M^F}
\newcommand\laFb{\lambda_F}
\newcommand\laFa{\lambda^F}
\newcommand\KFb{K_F}

\newcommand\PFb{P_F}

\newcommand\NFa{N^F}

\newcommand\pareteF[3]{\mathfrak{w}(#1;#2,#3)}

\newcommand\vbcFa{  \rlap{\v}{{\textbf{c}}}^F  }
\newcommand\Es[1]{\mathbf{E}_{#1}  }
\newcommand\Esi{\mathbf{E}_{\si}  }
\newcommand\EsFb[1]{(\mathbf{E}_F)_{#1}  }
\newcommand\Sspace[1]{Y(#1)}
\newcommand\Ss[2]{S_{\bE_{\sigma},{#2}}^{#1}}

\begin{document}

\title[Multipliers on symmetric spaces]
{Weak type estimates 
for spherical multipliers \\
on noncompact symmetric spaces }

\subjclass[2000]{}


\thanks{Work partially supported by the
Italian Progetto cofinanziato ``Analisi Armonica'' 2006--2008.}
% the \thanks layout looks crappy!

\author[S. Meda and M. Vallarino]
{Stefano Meda and Maria Vallarino}

%\address{Dipartimento di Matematica \\
%Universit\`a di Genova \\ via Dodecaneso 35, 16146 Genova \\ Italia}
\address{Stefano Meda:
Dipartimento di Matematica e Applicazioni
\\ Universit\`a di Milano-Bicocca\\
via R.~Cozzi 53\\ 20125 Milano\\ Italy -- stefano.meda@unimib.it}
\address{Maria Vallarino:
Laboratoire de Math\'ema\-ti\-ques et Applications, Physique Math\'ema\-ti\-ques
d'Orl\'eans\\
 Universit\'e d'Orl\'eans, UFR Sciences\\
B\^atiment de Math\'ematique-Route de Char\-tres\\
B.P. 6759\\ 45067 Orl\'eans cedex 2\\ France --  maria.vallarino@unimib.it}
 
\begin{abstract}
In this paper we prove sharp weak type $1$ estimates
for spherical Fourier multipliers on symmetric spaces of the
noncompact type.
This complements earlier results of J.-Ph.~Anker and
A.D.~Ionescu.
%and M.~Cowling, S.~Giulini and Meda.
\end{abstract}

\maketitle

\setcounter{section}{-1}
\section{Introduction} \label{s: Introduction}

The purpose of this paper is to give sharp weak type $1$ estimates
for a comparatively wide class of spherical Fourier multiplier operators
on Riemannian symmetric spaces of the noncompact type
that include the imaginary powers of the Laplace--Beltrami operator
$\cL$ and the resolvent operator $\cL^{-1}$.
Our result complements earlier results of J.-Ph.~Anker \cite{A1,A2} and
A.D.~Ionescu~\cite{I2,I3}, and may be thought of as an analogue
on noncompact symmetric spaces of
the classical Mihlin--H\"ormander multiplier theorem \cite{Ho}.

Suppose that $G$ is a noncompact semisimple Lie group with finite centre.
Denote by $K$ a maximal compact subgroup of $G$, and by $X$
the symmetric space of the noncompact type $G/K$.
We denote by $n$ and $\ell$ the dimension and the rank of $X$ respectively.
Denote by $\theta$ a Cartan involution
of the Lie algebra $\frg$ of $G$,
and write $\frg = \frk \oplus \frp$ for the corresponding
Cartan decomposition.  Let $\fra$ be a maximal
abelian subspace of $\frp$, and denote by $\fra^*$ its dual space,
and by $\astarc$ the complexification of $\fra^*$.
Denote by $\Sigma$ the set of (restricted) roots
of $(\frg,\fra)$; a choice for the set of positive roots
is written $\Sigma^+$, and $\fra^+$ denotes the corresponding
Weyl chamber.
The vector $\rho$ denotes
$(1/2)\sum_{\al \in \Sigma^+} m_{\al} \,\al$,
where $m_{\al}$ is the multiplicity of $\al$.
We denote by $\Sppp$ the set of simple
roots in $\Sigma^+$, and by $\Spp$ the set of indivisible
positive roots.  
Denote by $W$ the Weyl group of $(G,K)$, and by $\bW$ the interior of the
convex hull of the points $\{w\cdot \rho: w \in W\}$.
Clearly $\bW$ is an open convex polyhedron in $\fra^*$.
Recall that the Killing form $B(\cdot\kern1.5pt,\cdot)$ is a nondegenerate
bilinear form on~$\frg$ that is positive definite when
restricted to $\fra$.  This induces an inner product on $\fra^*$
and we denote by $\mod{\cdot}$ the associated norm.
Sometimes we shall use co-ordinates on~$\fra^*$.  When we do,
we always refer to the co-ordinates associated to the orthonormal basis
$\vep_1,\ldots,\vep_{\ell-1}, \rho/\mod{\rho}$, where 
$\vep_1,\ldots,\vep_{\ell-1}$ is any orthonormal basis of $\rho^\perp$.
In particular, for each multiindex $I= (i_1,\ldots,i_\ell)$, 
we denote by $D^I$
the partial derivative $\partial^{\mod{I}}/\partial_1^{i_1} 
\cdots \partial_\ell^{i_\ell}$ with respect to these
co-ordinates.

It is well known that $(G,K)$ is a Gelfand pair, i.e. the
convolution algebra $\lu{\KGK}$ of all $K$--bi-invariant
functions in $\lu{G}$ is commutative. The spectrum
of $\lu{\KGK}$ is the closure $\OTW$ in $\astarc$ 
of the tube $\TW = \fra^*+ i\bW$.
Denote by $\wt f$ the Gelfand transform
(also referred to as the spherical Fourier transform, or the
Harish-Chandra transform in~this~setting)
of the function $f$ in $\lu{\KGK}$.
It is known that $\wt f$ is a bounded continuous function on $\OTW$,
holomorphic in $\TW$,
and invariant under the Weyl group $W$.
The Gelfand transform
extends to $K$--bi-invariant
tempered distributions on $G$ (see, for instance, \cite[Ch.~6.1]{GV}).

For each $q$ in $[1,\infty)$,
denote by $\GOp{q}$ the Banach algebra of all $G$ invariant
bounded linear operators on $\laq{X}$, endowed with the
operator norm.
It is well known that $B$ is in~$\GOp{2}$ if and only
if there exists a $K$--bi-invariant tempered distribution $k_B$
on $G$ such that $\wt k_B$ is a bounded Weyl invariant function
on $\fra^*$ and
$$
Bf = f*k_B
\quant f \in \ld{X}
$$
(see \cite[Prop.~1.7.1 and Ch.~6.1]{GV} for details).
We call $k_B$ the \emph{kernel} of $B$.
We denote its spherical Fourier transform~$\wt k_B$
by $m_B$ and call it the \emph{spherical multiplier} associated to $B$.
As a consequence of a well known result of J.L.~Clerc and E.M.~Stein \cite{CS},
if $B$ is in~$\GOp{q}$ for all $q$ in $(1,\infty)$, then $m_B$
is a Weyl invariant holomorphic function in $\TW$,
bounded on closed substubes thereof.

For the rest of the Introduction we assume that $B$ is in
$\GOp{2}$ and that $m_B$ extends to a Weyl invariant holomorphic 
function in $\TW$, bounded on closed subtubes~thereof.
In this paper we consider the problem of finding conditions on $m_B$ such
that~$B$ extends to an operator of weak type $1$.

This problem has been considered by various authors.
Anker \cite{A1}, following up earlier results of M.~Taylor \cite{T}
and J.~Cheeger, M.~Gromov and Taylor \cite{CGT}
for manifolds with bounded geometry,
proved that if $m_B$ 
satisfies pseudodifferential estimates of the form
\begin{equation} \label{f: pseudodiff estimates}
\mod{D^{I} m_B(\zeta)}
\leq C\, \bigl(1+\mod{\zeta}\bigr)^{-\mod{I}}
\quant \zeta \in \TW
\end{equation}
for every multiindex $I$ such that $\mod{I} \leq [\![n/2]\!]+1$
($[\![\cdot]\!]$ denotes the integer part function),
then the operator $B$ is %bounded on $\lp{X}$ for all $p$ in $(1,\infty)$ and
of weak type $1$.
This extends previous results concerning special classes of symmetric spaces
\cite{CS,ST,AL}.

Anker's result was complemented by
A.~Carbonaro, G.~Mauceri and Meda \cite{CMM},
who showed that if $m_B$ satisfies (\ref{f: pseudodiff estimates}), then
$B$ is bounded from the Hardy space $H^1(X)$ to $\lu{X}$
and from $\ly{X}$ to the space $BMO(X)$ of functions of bounded
mean oscillation on $X$ (see \cite{CMM} for the definition
of these spaces). The space $BMO(X)$ had already been
defined in the rank one case in \cite{I1}, where an interesting application
to oscillatory multipliers is given.

These results are somewhat of ``local'' nature in the following
sense. If $m_B$ satisfies (\ref{f: pseudodiff estimates}), 
then the convolution kernel $k_B$
may be written as the sum of
a local part $k_B^{0}$, which has compact support near the origin
and satisfies standard Calder\'on--Zygmund type estimates,
and a part at infinity $k_B^\infty$, which is in $\lu{X}$
(see the proof of the main result in \cite{A1}).
Clearly, the convolution operator $f \mapsto f*k_B^\infty$ is bounded on
$\lu{X}$, hence of weak type~$1$.
Furthermore, a standard procedure reduces the problem
of proving weak type $1$ estimates for the convolution
operator $f \mapsto f* k_B^0$ to a similar problem where
$f$ is an $\lu{X}$ function supported near the origin.
Since $k_B^0$ satisfies a H\"ormander type integral condition,
the weak type $1$ estimate for $f \mapsto f* k_B^0$ follows
from the general theory of singular integrals
on spaces of homogeneous type 
in the sense of Coifman and Weiss \cite{CW,St1}.

In view of this remark
it is natural to consider the problem of finding fairly
general conditions on $m_B$ that are
strong enough to guarantee that $B$ extend to an
operator of weak type $1$ and nevertheless do not imply that
$k_B$ be integrable at infinity.

A result in this direction that improves the
aforementioned result of Anker may be obtained by
routine adaptation of methods of Ionescu~\cite{I2, I3}
and of J.-O.~Str\"omberg \cite{Str}.
Define the function $d: \OTW \to [0,\infty)$ by 
\begin{equation} \label{f: def d}
d(\xi+i \eta)
= \bigl[\mod{\xi}^2+ \mathrm{dist}(\eta, \bW^c)^2\bigr]^{1/2}
\quant \xi \in \fra^* \quant \eta \in \OV\bW.
\end{equation}
Suppose that $m_B$ 
satisfies H\"ormander--Mihlin type conditions of the form
\begin{equation} \label{f: Horm type estimates p}
\mod{D^{I} m_B(\zeta)}
\leq C\, d(\zeta)^{-\mod{I}}
\quant \zeta \in \TW
\end{equation}
for every multiindex $I$ such that $\mod{I} \leq N$, where $N$
is a sufficiently large integer.
Then the operator $B$ is 
%on $\lp{X}$ for all $p$ in $(1,\infty)$ and
of weak type $1$.
A careful analysis shows that the kernel $k_B$
may indeed be nonintegrable at infinity. 
See Section~\ref{s: Weak type} for the precise
statement of a sharper form of this result,
where we allow the multiplier $m_B$ itself to be
unbounded on~$\TW$.

Though interesting, this result is not completely
satisfactory, because in the higher rank case it 
does not apply to certain natural operators like
the purely imaginary powers of the Laplace--Beltrami
operator $\cL$ on $X$ (see Remark~\ref{remark: imaginary}
for details).  
Furthermore, observe that if $B$ is in $\GOp{2}$ and of weak type $1$,
then $m_B$ need not be bounded on~$\TW$.
For instance, for each complex number $\al$ such that
$0\leq \Re \al \leq 2$, 
the operator $\cL^{-\al/2}$, spectrally defined, 
is of weak type $1$ \cite{A2,AJ}, and
$$
m_{\cL^{-\al/2}} (\zeta)
= Q(\zeta)^{-\al/2}
\quant \zeta \in \TW
$$
is unbounded near the vertices of $0+i\bW$, in particular near $i\rho$.
Here $Q$ denotes the Gelfand transform of $\cL$ (see 
(\ref{f: autov Laplaciano}) 
and (\ref{f: def Q}) below).
Note that the weak type $1$ estimate for $\cL^{-\al/2}$ is derived
in \cite{A2,AJ} from sharp estimates for the heat kernel. 
It is unlikely that a similar strategy applies
to more general~multipliers. 

We aim at proving a multiplier result which applies to 
$m_{\cL^{-\al/2}}$ for all complex $\al$ with $0\leq \Re \al \leq 2$.
Given a multiindex $(I',i_\ell)$ in $\BN^\ell$,
where $I'$ is in $\BN^{\ell-1}$ and $i_\ell$ is in $\BN$,
denote by $\mod{I'}$ the length of $I'$.
For each $\kappa$ in $[0,\infty)$ consider the following
nonisotropic condition on the multiplier~$m_B$:
\begin{equation} \label{f: anisotropic MH condition}
\bigmod{D^{(I',i_\ell)} m_B(\zeta)}
\leq \frac{C}{\min\bigl(\mod{Q(\zeta)}^{\kappa+i_\ell+\mod{I'}/2},
\mod{Q(\zeta)}^{(i_\ell + \mod{I'})/2} \bigr)}
\quant \zeta \in \TWpiu,
\end{equation}
for all $(I',i_\ell)$ 
with $ \mod{I'}+i_\ell \leq [\![n/2]\!]+1$.
The set $\TWpiu$ is defined in Section~\ref{s: Notation}.
Our main result, Theorem~\ref{t: main result}, 
states that if $m_B$ satisfies (\ref{f: anisotropic
MH condition}) and either $\kappa$ is in $[0,1)$,
or $\kappa$ is $1$ and~$B$ is a spectral multiplier 
of $\cL$, then $B$ is of weak type $1$.
Theorem~\ref{t: main result}~\rmii\ is sharp and it is strong enough
to give the weak type $1$ boundedness of $\cL^{-\al/2}$
for all complex numbers $\al$ with $0\leq \Re \al \leq 2$.

We observe that in the higher rank case 
condition (\ref{f: anisotropic MH condition}) is new, even when 
$\kappa = 0$.  
It is straightforward to check that both $d(\zeta)$
and $\mod{Q(\zeta)}^{1/2}$ are equivalent to $\mod{\zeta}$ as $\zeta$
tends to infinity within the tube $\TW$.
Therefore both condition (\ref{f: Horm type estimates p}) 
and condition (\ref{f: anisotropic MH condition}) are equivalent
to condition (\ref{f: pseudodiff estimates}) at infinity.
Moreover, if $\ell =1$, then
$\mod{Q(\zeta)}$ and $\mod{\zeta -i\rho}$ are comparable
as~$\zeta$ tends to~$i\rho$, and condition (\ref{f: anisotropic
MH condition}) becomes
$$
\bigmod{m_B^{(j)}(\zeta)}
\leq \frac{C}{\min\bigl(\mod{Q(\zeta)}^{\kappa+j},
\mod{Q(\zeta)}^{j/2} \bigr)}
\quant \zeta \in \TWpiu.
$$
Hence conditions (\ref{f: anisotropic MH condition}) and
(\ref{f: Horm type estimates p}) are equivalent when $\ell=1$ and
$\kappa =0$.
We emphasise the fact that 
(\ref{f: anisotropic MH condition}) is not equivalent
to (\ref{f: Horm type estimates p}) when $\ell \geq 2$
and $\zeta$ tends to $i\rho$ within $\TW$.

Conditions analogous to (\ref{f: anisotropic MH condition}) 
but on tubes smaller than $\TW$
may be considered, and corresponding weak or strong type $p$
estimates for spherical multipliers may be proved.  To keep
the length of this paper reasonable we shall postpone the detailed
study of operators satisfying these conditions to a forthcoming paper. 

Our paper is organised as follows.
Section~\ref{s: Notation} contains some notation and terminology.
In Section~\ref{s: Weak type} we define
certain function spaces that appear in the
statement of our main result, and state Theorem~\ref{t: main result}.
Sections~\ref{s: The higher rank case I} and \ref{s: Kernel estimates} 
are quite technical.  
In Section~\ref{s: The higher rank case I} we adapt
methods of Str\"omberg \cite{Str} to prove weak type $1$
boundedness results for the convolution operators with kernels
which are relevant in the proof of Theorem~\ref{t: main result}
(see formula (\ref{f: estimatetau})). 
Section~\ref{s: Kernel estimates} is devoted to estimating
the kernel $k_B$ when $m_B$ satisfies (\ref{f: anisotropic
MH condition}).
The proof of Theorem~\ref{t: main result} hinges on
the results of Sections~\ref{s: The higher rank case I}
and \ref{s: Kernel estimates}, and is given in
Section~\ref{s: Proof}.

We will use the ``variable constant convention'', and denote by $C,$
possibly with sub- or superscripts, a constant that may vary from place to
place and may depend on any factor quantified (implicitly or explicitly)
before its occurrence, but not on factors quantified afterwards.

\section{Notation and background material} \label{s: Notation}

We use the standard notation of the theory of Lie groups and
symmetric spaces, as in the books of Helgason~\cite{H1,H2}. 
We shall also refer to the book \cite{GV} and to the
paper \cite{AJ}.

In addition to the notation above, denote by $\frn$ the subalgebra
$\sum_{\al \in \Sigma^+} \frg_{\al}$ of $\frg$.
By $N,$ $\OV N,$ $A,$ and $K$ we denote the subgroups of $G$ corresponding
to $\frn,$ $ \theta \frn,$
$\fra,$ and $\frk$ respectively, and write $G = KAN$ and $G = \OV NAK$
for the associated Iwasawa decompositions.
Given $\la$ in $\fra^*$, define $H_\la$
to be the unique element in $\fra$ such~that
$$
B(H_\la ,H)=\la (H) \quant H \in \fra,
$$
and then an inner product $\prodo{\cdot\kern1.5pt}{\cdot}$
on $\fra^*$ by the rule
$$
\prodo\la{\la'} = B(H_\la ,H_{\la'})
\quant \la ,\la' \in \fra^*.
$$
We abuse the notation, and denote by $\mod{\cdot}$
both the norms associated to the inner products $\prodo{\cdot}{\cdot}$
on $\fra^*$ and $B(\cdot,\cdot)$ on $\fra$.
The inner product $\prodo{\cdot}{\cdot}$ on $\fra^*$ extends to a bilinear
form, also denoted $\prodo{\cdot\kern1.5pt}{\cdot},$ on $\fra_{\BC}^*$.
For any $R$ in $\BR^+$ define
\begin{equation} \label{f: Ball}
\Ball{R} 
= \{\la \in \fra^* : \mod{\la} < R \}.
\end{equation}
The ball $\Ball{\mod{\rho}}$ will occur frequently in the
analysis of functions of the Laplace--Beltrami operator.
For notational convenience, we shall write
$\bB$ instead of $\Ball{\mod{\rho}}$.

If $H$ is in $\fra$, we write $(H_1,\ldots,H_\ell)$ for
the vector of its co-ordinates with respect to the dual basis
of the basis $\vep_1,\ldots,\vep_{\ell-1},\rho/\mod{\rho}$ 
of $\fra^*$ defined in the Introduction.  Observe that the last vector of
this dual basis is $H_\rho/\mod{H_\rho}$.  
Sometimes we shall write $H'$
instead of $(H_1,\ldots, H_{\ell-1})$.  
Define $\cN: \fra \to \BR$ by
\begin{equation} \label{f: anisotropic norm}
\cN(H',H_\ell)
= \bigl(\mod{H'}^4+H_\ell^2\bigr)^{1/4}.
\end{equation}
Note that~$\cN$ is homogeneous with respect to the dilations
$
(H',H_\ell) \mapsto (\vep \, H', \vep^2 \, H_\ell),
$
and that the homogeneous dimension of~$\fra$, endowed with the
(quasi) metric induced by $\cN$, is $\ell + 1$.
Suppose that $R$ is in $\BR^+$.  Define
\begin{equation}  \label{f: NBall}
\NBall{R} 
 = \{H \in \fra: \cN(H',H_\ell) < R \}.
\end{equation}
Define the parabolic region $\Parabolic$ in $\fra$ by 
\begin{equation} \label{f: Parabolic}
\Parabolic
= \{ H\in \fra : \mod{H'} < H_\ell^{1/2} \}.
\end{equation}
Define the functions $\om: \fra \to \BR$ and
$\om^*: \fra^*\to \BR$ by 
\begin{equation} \label{f: omega}
\om(H)
= \min_{\al\in\Sppp} \al(H)  \quant H \in \fra
\qquad\hbox{and}\qquad
\om^*(\la) 
= \min_{\al\in\Sppp} \prodo{\al}{\la} \quant \la\in\fra^*.
\end{equation}
Furthermore for each
$c$ in $\BR^+$, define the subset $\Strip_c$ of $\OV{\fra^+}$ by
\begin{equation} \label{f: Strip}
\Strip_c = \{ H \in \fra: 0\leq \om(H) \leq c \}.
\end{equation}
Denote by $\bigl(\fra^*\bigr)^+$ the interior of the
fundamental domain of the action of the Weyl group $W$
that contains $\rho$.
For any subset $\bE$ of $\fra^*$ denote by 
$T_{\bE}$ the tube over $\bE$, i.e., the set $\fra^*+i \bE$ 
in the complexified space 
$\fra_\BC^*$,  and by $\OV{T}_{\bE}$ its closure in $\fra_\BC^*$.
For each $t$ in $\BR$ we denote by $\bE^t$ the set 
\begin{equation} \label{f: ngbhd of E+}
\bE^t
= \{ \la \in \bE: \, \,  \om^*(\la) > t \}.
\end{equation}
Note that if $E$ is open, then $\bE^0$ is the interior
of $\OV{(\fra^*)^+}\cap \bE$.  For simplicity, we shall write
$\bE^+$ instead of $\OV{(\fra^*)^+}\cap \bE$.
Notice that $\bW^+$ is neither open nor closed in $\fra^*$, whereas
for each $t$ in $\BR^-$ the set $\bW^t$ is an 
open neighbourhood of $\bW^+$ that contains the origin. 
Thus, $\TWpiut$ is a neighbourhood
of $\TWpiu$ in $\fra_\BC^*$ that contains $\fra^*+i 0$.

We write $\wrt x$ for a Haar measure on $G,$ and let $\wrt k$ be the
Haar measure on $K$ of total mass one.
We identify functions on the symmetric space $X$ with
right--$K$--invariant functions on~$G$, in the usual way.
If $E(G)$ denotes a space of functions on $G$, we
define $E({\KX})$ and $E({X})$ to be the closed subspaces of
$E(G)$ of the $K$--bi-invariant and the right--$K$--invariant
functions respectively.
The Haar measure of $G$ induces a $G$--invariant measure
$\wrt \dot x$ on $X$ for~which
$$
\int_X f(\dot x) \, \wrt \dot x = \int_G f(x) \, \wrt x
\quant f \in C_c(X),
$$
where $\dot x = xK.$
We recall that
$$
\int_G f(x)\wrt x 
= \int_K\!\int_{\fra^+}\!\int_K 
f\bigl(k_1(\exp H) k_2\bigr)\, \de(H) \wrt k_1\wrt k_2 \wrt H,
$$
where $\wrt H$ denotes a suitable nonzero multiple of the Lebesgue
measure on $\fra$, and 
$$
\de(H)=
\prod_{\al\in\Sp}\sinh^{m_{\al}}\bigl(\al(H)\bigr)
\leq C\, \e^{2\rho(H)}
\quant H \in \fra^+.
$$
For any $a$ in $A$
we denote by $\log a$ the element $H$ in $\fra$ such that $\exp H=a$.
For any $x$ in $G$, we denote by $H(x)$ the unique element of $\fra$
such that $x$ is in $K \exp H(x) N$. Thus, $H(kan) = \log a$.
For any $\la$ in $\astarc$, 
the elementary spherical function $\vp_\la$ is defined by the rule
$$
\vp_\la (x) = \int_K \exp [(i\la - \rho) H(xk)] \wrt k
\qquad \forall x \in G.
$$
The spherical transform $\sft f$, also denoted by $\cH f$,
of an $\lu G$-function $f$ is
defined by 
$$
\sft f (\la) = \int_G f (x) \, \vp_{-\la} (x) \wrt x
\qquad \forall \la \in \fra^*.
$$
Harish-Chandra's inversion formula and Plancherel formula state
that
$$
f(x) = \int_{\fra^*} \sft f (\la) \, \vp_\la(x)\, \planl
\qquad \forall x \in G
$$
for ``nice'' $K$-bi-invariant functions $f$ on $G$, and
$$
\norm{f}2 = \left[\int_{\fra^*} \bigmod{\sft f(\la)}^2\, \planl \right]^{1/2}
\qquad\forall f\in \ld {K\backslash G/K},
$$
where $\planl = c_{{}_G} \planc $, and $\mathbf{c}$ denotes the
Harish-Chandra $\mathbf{c}$-function. For the details, see, for instance,
\cite[IV.7]{H1}.  
Sometimes we shall write $\cH^{-1}$ for the inverse
Fourier transform. 
The Harish-Chandra $\mathbf{c}$-function is given by
$$
\bc(\la) = \prod_{\al\in\Spp}  c_{\al}(\prodo{\al}{\la}),
$$
where each Plancherel factor 
$c_{\al}$ is given by an explicit formula involving several 
$\Ga$-functions \cite[Thm 6.14]{H1}. It is well kwnown that
\begin{equation}\label{f: estimatec}
\mod{\bc(\la)}^{-2} 
\leq C\, \bigl(1+\mod{\la}   \bigr)^{\sum_{\al\in\Spp}d_{\al}}
\leq C\, \bigl(1+\mod{\la}   \bigr)^{n-\ell},
\end{equation}
where $d_{\al}=\rm{dim} \mathfrak{g}_{\al}+ \rm{dim} \mathfrak{g}_{2\al}$. 
We denote by $\vbc$ the function $\vbc(\la)=\bc(-\la)$ 
which is holomorphic in $\TWpiut$ for 
some negative $t$ and satisfies the following estimate
$$
\mod{ (\vbc)^{-1}(\zeta) }\leq C\, \prod_{\al\in\Spp} 
\bigl(1+\mod{\zeta}   \bigr)^{\sum_{\al\in\Spp}d_{\al}/2}
\leq C\, \bigl(1+\mod{\zeta}   \bigr)^{(n-\ell)/2}
\quant\zeta\in \TWpiut.
$$
This, the analyticity of $(\vbc)^{-1}$ on $\TWpiut$, and Cauchy's
integral formula imply that  for every multiindex $I$  
\begin{equation}  \label{f: derivata c check}
\bigmod{D^I(\vbc)^{-1}(\zeta)}
\leq C\, \bigl(1+\mod{\zeta}   \bigr)^{(n-\ell)/2}
\quant \zeta\in \TWpiut.
\end{equation}

Now, we describe the various faces of $\OV{\fra^+}$ which are 
in one-to-one correspondence with the nontrivial
subsets $F$ of $\Sppp$. 
We denote by $\SpFb$ the positive root subsystem 
generated by $F$ and by $\SppFb$ the positive indivisible 
roots in $\SpFb$.  Then we may write
$$
\fra=\aFb\oplus\aFa,\qquad \fra^*=\aFb^*\oplus \astFa,
\qquad \frn=\nFb\oplus\nFa \qquad {\rm{and~}} \qquad N=N_FN^F,
$$
where $\aFb$ is the subspace generated by the vectors $\{H_\al:
\al \in F \}$, $\aFa$ denotes its orthogonal 
complement in $\fra$, $\aFb^*$ is the subspace of $\fra^*$ 
generated by $F$, $\astFa$ denotes its orthogonal complement
in $\fra^*$, $\nFb=\oplus_{\al\in\SpFb}\frg_{\al}$ and 
$\nFa=\oplus_{\al\in\Sp\setminus \SpFb}\frg_{\al}$. 
The face $(\aFa)^+$ of $\OV{\fra^+}$ attached to $F$~is
$$
(\aFa)^+ = 
\{ H\in\aFa:  \al(H)>0 \, \, \, \, \,  \forall\al\in\Sppp\setminus F  \}.
$$
We shall write $H=\HFb+\HFa$ and $\la=\laFb+\laFa$ 
according to the decompositions 
$\fra=\aFb\oplus\aFa$ and $\fra^*=\aFb^*\oplus\astFa$
respectively.  In particular, $\rho=\rFb+\rFa$. 
Observe that $\ell=\lFb+\lFa$, where $\lFb$ and $\lFa$ denote the 
dimensions of $\aFb$ and $\aFa$, respectively.

We denote by $\La$ the lattice $\sum_{\al\in\Sppp}\mathbb{N}\al$. 
Observe that $\La=\QFb+\QFa$, 
where $\QFb = \sum_{\al\in F}\BN\al$ and 
$\QFa = \sum_{\al\in\Sppp\setminus F}\BN\al$, and
$$
\bc
= \bcFb\,\, \bcFa,
$$
where 
$$
\begin{aligned}
\bcFb(\la)
&  =\prod_{\al\in\SppFb}
c_{\al}(\prodo{\al}{\la})
\qquad {\rm{and}}\qquad
\bcFa(\la)
=\prod_{\al\in\Spp\setminus \SppFb}
c_{\al}(\prodo{\al}{\la}).
\end{aligned}
$$
We shall often use the following estimates:
\begin{equation}\label{f: estimatecFbcFa}
\mod{\bcFb(\la)}^{-2} 
\leq C\, \bigl(1+\mod{\la}   \bigr)^{\sum_{\al\in\SppFb}d_{\al}}
\qquad 
\mod{D^I_{\la}(\vbcFa)^{-1}(\la)} 
\leq C\, \bigl(1+\mod{\la}   \bigr)^{\sum_{\al\in\Spp\setminus \SppFb}d_{\al}/2}
\end{equation}
and for every multiindex $I$
\begin{equation}\label{f: estimatecFbcFainsieme}
\mod{\bcFb(\la)}^{-1} \,\mod{D^I_{\la}(\vbcFa)^{-1}(\la)}
\leq C\, \bigl(1+\mod{\la}   \bigr)^{n-\ell}.
\end{equation}

We denote by $\PFb$ the normalizer of $\NFa$ in $G$; it has Langlands 
decomposition $\PFb=\MFb(\exp\aFa)\NFa$, where $\MFb$ and 
$\MFa=\MFb(\exp\aFa)$ are closed subgroups of $G$. 
We denote by $\omFaast$ and $\omFbast$ the functions defined by
$$
\omFaast(\la)
= \min_{\al\in\Sppp\setminus F}\prodo{\al}{\la}
\qquad\hbox{and}
\qquad
\omFbast(\la)
= \min_{\al\in F}\prodo{\al}{\la}
\quant \la\in\fra^*.
$$
The height of an element $q=\sum_{\al\in\Sppp}n_{\al}\al$ 
in $\La$ is defined by
$
\height{q}=\sum_{\al\in\Sppp}n_{\al}.
$
The asymptotic expansion of the spherical functions along the walls of the Weyl 
chamber is due to P.C.~Trombi and V.S.~Varadarajan \cite[Thm~2.11.2]{TV}
(see also \cite[Thm~5.9.4]{GV}).  For the reader's convenience
we state the following variant of \cite[Thm~2.11.2]{TV}, due to Anker and Ji
\cite[Theorem 2.2.8]{AJ}.

\begin{theorem}\label{t: asymptoticexpansion}
Suppose that $X$ is a symmetric space of the noncompact type. 
Suppose that $F$ is a nontrivial
subset of $\Sppp$, that $\la$ is regular and that $H$ is in $\OV{\fra^+}$ 
with $\omFa(H)>0$. We have an asymptotic expansion 
$$
\vp_{\la}(\exp H)\sim 
\e^{-\rFa(H)}\,\sum_{q\in \QFa}\,\sum_{w\in \WFb \backslash W}
\bcFa(w\cdot\la)\,\phiFa{w\cdot\la}{q}(\exp H),
$$
where
\begin{enumerate}
\item[\itemno1] $\phiFa{\la}{0}$ is the spherical function 
of index $\la$ on $\MFa=\MFb\,\exp \aFa$ and
$$
\phiFa{\la}{0}(x)=\vp^F_{\laFb}(y)\,\e^{i\laFa(H)}
\quant x=y\exp H\in \MFb\exp\aFa;
$$ 
\item[\itemno2] $\phiFa{\la}{q}$ are bi-$\KFb$-invariant $C^{\infty}$ 
functions in the variable $x\in\MFa$ and $\WFb$-invariant holomorphic 
functions in the variable $\la$ in the region
$$   
\{\la = \laFb + \laFa \in \astarc:
\mod{\Im\laFb}<c,\,\omFaast(\Im\laFa)>-c  \},
$$
for some small positive $c$; moreover,
$$
\phiFa{\la}{q}(x)
= \phiFa{\la}{q}(y)\,\e^{(i\la-q)(H)}
\quant x=y\exp H\in \MFb\exp\aFa;
$$ 
\item[\itemno3] for every $q$ in $\QFa$ there exists a constant $d\geq 0$ and 
for every positive $c$ there exists a constant $C\geq 0$ such that
$$
\bigmod{\phiFa{\la}{q}(\exp H)}\leq C\,\e^{c\,\mod{\HFb}}\,
(1+\mod{\la})^{d} \,
\e^{-[\Im(\la)+\rFb+q](H)}
\quant \la\in\fra^*+i\OV{(\astFa)^+}, \,H\in\OV{\fra^+};
$$
\item[\itemno4] for every positive integer $N$ there exists a constant $d\geq 0$ 
and for every positive $c$ there exists a constant $C\geq 0$ such that
$$
\begin{aligned}
&\Bigmod{\vp_{\la}(\exp H)-  \e^{-\rFa(H)}\,
\sum_{q\in \QFa,\,\height{q}<N }\,\sum_{w\in \WFb \backslash W}
\bcFa(w\cdot\la)\,\phiFa{w\cdot\la}{q}(\exp H)  }\\
\leq & \,C\, (1+\mod{\la})^d\,(1+\mod{H})^d\, \e^{-\rho(H)-N\omFa(H)}
\end{aligned}
$$
for $\omFa(H)>c$.
\end{enumerate}
\end{theorem}

Denote by $\cL_0$ minus the Laplace--Beltrami operator on $X$
associated to the metric given by the Killing form on $\frg$:
$\cL_0$ is a symmetric operator on $C_c^\infty(X)$ (the space of
smooth complex-valued functions on $X$ with compact support).  Its closure
is a self adjoint operator on $\ld{X}$ that we denote by $\cL$.
It is known that 
the bottom of the $\ld{X}$ spectrum of $\cL$
is~$\prodo{\rho}{\rho}$.
Note that
\begin{equation} \label{f: autov Laplaciano}
\cL \vp_\la = Q(\la)\, \vp_\la
\quant\la\in \fra_{\BC}^*,
\end{equation}
where $Q$ is the quadratic function on $\astarc$
defined by
\begin{equation} \label{f: def Q}
Q(\zeta) = \prodo{\zeta}{\zeta} + \prodo{\rho}{\rho}
\quant \zeta \in \astarc.
\end{equation}
The operator $\cL$ generates a symmetric diffusion
semigroup $\{\cH^t\}_{t>0}$ on $X$. For $t$ in $\Rplus ,$
denote by $h_t$ the heat kernel at time $t$, i.e., 
\begin{equation} \label{f: heat kernel}
h_t(x)
= \int_{\fra^*} e^{-t\, Q(\la)} \, \vp_\la(x)\, \planl
\quant x \in G.
\end{equation}

\section{Statement of the main result}
\label{s: Weak type}

In this section we define some Banach spaces of holomorphic
functions that are relevant for our analysis of spherical
multipliers, and study their relationships.  Then we state
our main result.

The following definition is motivated by the main result in \cite{I2,I3}.

\begin{definition} \label{def: MHI conditions}
Suppose that $J$ is a nonnegative integer and that
$\kappa$ is in $[0,\infty)$. We denote by $\wt H(\TW;J,\kappa)$ 
the space of all holomorphic functions $m$ in $\TW$ such that
$\norm{m}{\wt H(\TW;J,\kappa)}< \infty$, where
$\norm{m}{\wt H(\TW;J,\kappa)}$ is the infimum of all constants $C$
such that 
\begin{equation}\label{MHconditions}
\mod{D^{I} m(\zeta)}
\leq \frac{C}{\min\bigl(d(\zeta)^{\kappa+\mod{I}},
d(\zeta)^{\mod{I}} \bigr)} 
\quant {\zeta \in \TW}\quant{I : \mod{I} \leq J}
\end{equation}
and $d$ is defined in (\ref{f: def d}).
\end{definition}

The following result complements the work of Ionescu \cite{I2,I3}.
Recall that $n$ and $\ell$ denote the dimension and the rank of $X$
respectively.

\begin{theorem} \label{t: weak1}
Assume that $\kappa$ is in $[0,1)$.
Suppose that $B$ is an operator in $\GOp{2}$ and that
$m_B$ is in $\wt H\bigl(\TW;[\![n/2]\!]+\ell/2+1,\kappa\bigr)$.
Then~$B$ extends to an operator of weak type $1$.
\end{theorem}

\begin{proof}
The proof of this theorem is rather long and technical,
and follows the lines of the proof of the main
result in \cite{I3}.  We omit the details, because we are more
interested in a different condition on the multipliers. 
\end{proof}

\begin{remark} \label{remark: imaginary}
Note (see \cite{I2}) that if $\ell=1$ and $\kappa = 0$, then 
Theorem~\ref{t: weak1} applies to
the multiplier $m_{\cL^{iu}}$, when $u$ is real.
However, if $\ell \geq 2$, then the multiplier $m_{\cL^{iu}}$ 
does not belong to $\wt H(\TW;J,\kappa)$ for any $\kappa$ in $[0,1]$.
We prove this in the case where $\kappa = 0$.

Indeed, suppose that $\Re(\zeta)$ is small.  A straightforward
computation shows that 
\begin{equation} \label{f: imaginary powers 1}
d(\zeta) \, \mod{\partial_{\zeta_\ell} m_{\cL^{iu}}(\zeta)}
= 2\, \mod{u} \, d(\zeta) \, \frac{\mod{\zeta_\ell}}{\mod{Q(\zeta)}}
\, \e^{-u \arg Q(\zeta)}
\quant \zeta\in \TW.
\end{equation}
Here $\zeta = (\zeta_1,\ldots,\zeta_\ell)$, and 
$\zeta_1,\ldots,\zeta_\ell$ are the
co-ordinates described in the Introduction.
We show that if $\ell\geq 2$, then the right hand side cannot
possibly stay bounded when $\zeta$ tends to $i\rho$ in $\OTW$.
Write $\zeta = \xi + i \eta$, where $\xi$ is in $\fra^*$
and $\eta$ is in $\bW$. Suppose that $\xi \neq 0$, and let
$\eta$ tend to $\rho$ within~$\bW$. By continuity, the right
hand side of (\ref{f: imaginary powers 1}) tends to
$$
2 \, \mod{u} \, d\bigl(\xi+i \, \rho\bigr) \,
\frac{\mod{\xi_\ell+i\, \rho}}{\mod{Q\bigl(\xi+i\, \rho\bigr)}}
\, \e^{-u \arg Q(\xi+i\eta)}.
$$
Now, $d(\xi+i\, \rho) = \mod{\xi}$ and
$Q\bigl(\xi+i \, \rho\bigr)
= \mod{\xi}^2 + 2i\, \prodo{\xi}{\rho}$.  Therefore,
if $\xi$ is orthogonal to $\rho$, then the right hand side of
(\ref{f: imaginary powers 1}) becomes
$
2 \, \mod{u} \, \mod{\xi} \, {\mod{\rho}}/{\mod{\xi}^2},
$
which tends to infinity when $\xi$ tends to $0$, as required.
\end{remark}

Denote by $\bP$ the parabolic region in the plane defined by
$$
\bP = \{(x,y)\in \BR^2: y^2 < 4 \prodo{\rho}{\rho} \, x \}.
$$
Note that $\bP$ is the image of $\TW$ under $Q$.
If $M(\cL)$ is in $\GOp{q}$ for
all $q$ in $(1,\infty)$, then its spherical multiplier 
$M\circ Q$ is holomorphic in~$\TW$ by the Clerc--Stein condition,
and $M$ is holomorphic in $\bP$.
This partially motivates the definition below. 

\begin{definition}
Suppose that $J$ is a nonnegative 
integer and that $\kappa$ is in $[0,\infty)$.
Denote by $\fH (\bP;J,\kappa)$
the space of all holomorphic functions $M$ in~$\bP$
such that $\norm{M}{\fH(\bP;J,\kappa)} < \infty$, where
$\norm{M}{\fH(\bP;J,\kappa)}$ is the infimum of all constants $C$
such that
$$
\mod{M^{(j)}(z)}
\leq \frac{C}{\min\bigl(\mod{z}^{\kappa+j}, \mod{z}^j \bigr)}
\quant {z\in \bP} \quant j \in \{0,1,\ldots, J\}.
$$
\end{definition}

Clearly for each $\be$ such that $\Re \be \geq 0$
the function $z\mapsto z^\be$ is in $\fH(\bP;J, \Re\be)$
for all $J\geq 0$. 
Note that if $M$ is holomorphic in $\bP$, then $M\circ Q$ is, in fact,
Weyl invariant and holomorphic in $T_{\bB}$.
In Proposition~\ref{p: McircQ0} below we prove that if 
$M$ is in $\fH(\bP;J,\kappa)$, then $M\circ Q$ is in  the 
space $H(\TB;J, \kappa)$, which we now define.  

\begin{definition} \label{def: moltiplicatori II}
Suppose that $J$ is a positive integer, $\kappa$ is in $[0,\infty)$,
and assume that $\bE$ is a convex neighbourhood of the origin in $\fra^*$.
Denote by $H(\TE;J,\kappa)$ the space of
all holomorphic functions $m$ in $\TE$ for which
$\norm{m}{H(\TE;J,\kappa)}<\infty$, where
$\norm{m}{H(\TE;J,\kappa)}$ is the infimum of all constants
$C$ such that 
$$
\bigmod{D^{I} m(\zeta)}
\leq \frac{C}{ \min \bigl(\mod{Q(\zeta)}^{\kappa+\mod{I}},
\mod{Q(\zeta)}^{\mod{I}/2}\bigr)}
\quant{\zeta\in \TEpiu} \,
\quant {I: \mod{I} \leq J}.
$$
See Section~\ref{s: Notation}
for the definition of $\bE^+$.
\end{definition}

In the rest of the paper we shall consider spaces
$H(\TE;J,\kappa)$ when $\bE$ is either $\Ball{\phantom{}}$ or
$\Ball{\phantom{}}^t$ for some $t$ in $\BR^-$.

\begin{proposition} \label{p: McircQ0}
Suppose that $J$ is a nonnegative integer and that $\kappa$
is in $[0,\infty)$.
Then there exists a constant $C$ such that 
$$
\norm{M\circ Q}{H(\TB;J,\kappa)}
\leq C \, \norm{M}{\fH(\bP;J, \kappa)}.
\quant M \in \fH(\bP;J,\kappa).
$$
\end{proposition}

\begin{proof}
Suppose that $I$ is a multiindex.
A straightforward induction argument shows that
there exist constants $c_P$ such that
\begin{equation} \label{f: induction}
D^I (M\circ Q)(\zeta)
= \sum_{0\leq P\leq I/2} c_P\,\, \zeta^{I-2P}\,
M^{(\mod{I}-\mod{P})}\bigl(Q(\zeta)\bigr)
\quant \zeta \in \TBpiu.
\end{equation}
Observe that if $\zeta$ is bounded, then so is $\mod{Q(\zeta)}$.
Since $M$ is in $\fH(\bP;J,\kappa)$,
$$
\begin{aligned}
\mod{D^I (M\circ Q)(\zeta)}
& \leq C\, \norm{M}{\fH(\bP;J,\kappa)} \,
            \sum_{0\leq P\leq I/2} \mod{\zeta}^{\mod{I}-2\mod{P}}\,
            \bigmod{Q(\zeta)}^{-\kappa-\mod{I}+\mod{P}}\\
& \leq C\, \norm{M}{\fH(\bP;J,\kappa)} \,\sum_{0\leq P\leq I/2}
            \bigmod{Q(\zeta)}^{-\kappa-\mod{I}+\mod{P}}\\
& \leq C\,\norm{M}{\fH(\bP;J,\kappa)} \,\bigmod{Q(\zeta)}^{-\kappa-\mod{I}}
\quant\zeta\in \Ball{\phantom{}}+i\bB^+.
\end{aligned}
$$
If, instead, $\zeta \in \bigl(\fra^*\setminus \Ball{\phantom{}}\bigr)+i\bB^+$,
then $\mod{Q(\zeta)} \geq C\,\mod{\zeta}^2$
for some positive constant $C$.  Hence 

$$
\begin{aligned}
\mod{D^I (M\circ Q)(\zeta)}
& \leq C\, \norm{M}{\fH(\bP;J,\kappa)} \,
            \sum_{0\leq P\leq I/2} \mod{\zeta}^{\mod{I}-2\mod{P}}\,
            \bigmod{Q(\zeta)}^{-\mod{I}+\mod{P}}\\
& \leq C\,\norm{M}{\fH(\bP;J,\kappa)} \,\bigmod{Q(\zeta)}^{-\mod{I}/2}
\quant\zeta\in \bigl(\fra^*\setminus \Ball{\phantom{}}\bigr)+i\bB^+.
\end{aligned}
$$
Thus, $M\circ Q$ is in $H(\TB;J,\kappa)$ and 
$\norm{M\circ Q}{H(\TB;J,\kappa)}\leq C\,\norm{M}{\fH(\bP;J,\kappa)}$, 
as required.
\end{proof}

In the higher rank case most spherical multipliers 
are not of the form $M\circ Q$ with $M$ holomorphic in $\bP$,
and, in general do not extend to holomorphic functions
in a region larger than $\TW$. 
We would like to prove a result which applies
to multipliers of the form $m\, (M\circ Q)$,
where 
$M$ is in $\fH(\bP;J,\kappa)$, and that
$m$ is holomorphic and bounded in $\TW$ and satisfies
estimates (\ref{f: pseudodiff estimates}).
To introduce the appropriate function space
we need more notation. 
For every multi-index $I = (i_1,i_2,\ldots,i_\ell)$
we shall denote by $D^I$ the differential operator
$
\partial_{\zeta_1}^{i_1} \partial_{\zeta_2}^{i_2}
\cdots \partial_{\zeta_\ell}^{i_\ell},
$
where $\zeta = \xi+i \eta$, $\xi$ and $\eta$ are in $\fra^*$,  
$\zeta_j = \xi_j + i \eta_j$, and $(\xi_1, \ldots, \xi_\ell)$ and
$(\eta_1, \ldots, \eta_\ell)$ are the co-ordinates of $\xi$ and $\eta$
with respect to the basis $\vep_1, \ldots, \vep_{\ell-1}, \rho/\mod{\rho}$,
defined in the Introduction.  
%Then we denote by $\mod{I}$ the \emph{length} $i_1+i_2+\cdots+i_\ell$
%of the multi-index $I$, and by $\mod{I}'$
%the \emph{nonisotropic length} $2i_1+i_2+\cdots+i_\ell$
%of $I$.

\begin{definition} \label{def: moltiplicatori}
Suppose that $J$ is a positive integer and that $\kappa$ is in $[0,\infty)$,
and assume that $\bE$ is a convex neighbourhood of the origin 
in $\fra^*$.
Denote by $H'(\TE;J,\kappa)$ the space of
all holomorphic functions $m$ in $\TE$ for which
$\norm{m}{H'(\TE;J,\kappa)}<\infty$, where
$\norm{m}{H'(\TE;J,\kappa)}$ is the infimum of all constants $C$
such that 
$$
\bigmod{D^{(I',i_\ell)} m(\zeta)} %\norm{m}{H'(\TW;J,\kappa)}
\leq \frac{C}{\min \bigl(\mod{Q(\zeta)}^{\kappa+i_\ell +\mod{I'}/2}, 
\mod{Q(\zeta)}^{(\mod{I'}+i_\ell)/2}\bigr) }
\quant \zeta\in \TEpiu
$$
for all multiindices $(I',i_\ell)$ for which $\mod{I'}+i_\ell \leq J$.
\end{definition}

In the rest of the paper we shall consider spaces 
$H'(\TE;J,\kappa)$, where $\bE$ is either $\bW$, or~$\bW^t$
for some $t$ in $\BR^-$.
Observe that the functions in $H'(\TWpiut;J, \kappa)$ satisfy on $\TWpiu$
the same estimates that functions in $H'(\TW;J,\kappa)$ satisfy,
but they need not be holomorphic in the whole tube $\TW$.
A similar observation applies to functions in the spaces
$H(\TB;J, \kappa)$ and $H(\TBpiut;J, \kappa)$ defined above.

\begin{remark} \label{rem: nonisotropic cond}
Suppose that $m$ is in $H'(\TW;J,\kappa)$ and that
the function $\xi \mapsto m(\xi+i \rho)$
is smooth on $\fra^* \setminus\{0\}$. By a continuity argument
for each multiindex $(I',i_\ell)$ with $\mod{I'}+i_\ell \leq J$ the function
$m$ satisfies
\begin{equation} \label{f: AMH 1}
\mod{D^{(I',i_\ell)} m(\xi+i\rho)} \leq
\frac{{\norm{m}{H'(\TW;J,\kappa)}}}{\min\bigl(\mod{
Q(\xi+i\rho)}^{\kappa+i_\ell+ \mod{I'}/2},
\mod{Q(\xi+i\rho)}^{(i_\ell+\mod{I'})/2}\bigr)}
\quant \xi\in \fra^*\setminus\{0\}.
\end{equation}
Note that
$\min \bigl(\mod{Q(\zeta)}^{\kappa+i_\ell+\mod{I'}/2},
\mod{Q(\zeta)}^{(i_\ell+\mod{I'})/2}\bigr)$
is equal to $\mod{Q(\zeta)}^{\kappa+i_\ell+\mod{I'}/2}$ if 
$\mod{\zeta}$ is small
and to $\mod{Q(\zeta)}^{(i_\ell+\mod{I'})/2}$ if $\mod{\zeta}$ is large.
Furthermore $\mod{Q(\xi+i \rho)} =\mod{\mod{\xi}^2+2i \, \prodo{\xi}{\rho}}$.
Thus,
$$
\mod{Q(\xi+i\rho)}
\asymp \left\{
\begin{array}{ll}
\mod{\xi}^2 & \hbox{if $\xi$ is either large, or small and $\xi\perp\rho$} \\
\mod{\xi} & \hbox{if $\xi = c\, \rho$ for $c\in \BR^+$ small.}
\end{array}
\right.
$$
Then, fr{}om (\ref{f: AMH 1}) we deduce that
$$
\mod{D^{(I',i_\ell)} m(\xi+i\rho)} \leq 
\left\{
\begin{array}{ll}
{\norm{m}{H'(\TW;J,\kappa)}} \,\mod{\xi}^{-\mod{I'}-i_\ell} 
          & \hbox{if $\xi$ is large} \\
{\norm{m}{H'(\TW;J,\kappa)}} \,\mod{\xi}^{-(\kappa+i_\ell+\mod{I'}/2)}
           & \hbox{if $\xi = c\, \rho$ for $c\in \BR^+$ small}\\
{\norm{m}{H'(\TW;J,\kappa)}} \,\mod{\xi}^{-(2\kappa+2i_\ell + \mod{I'})}
           & \hbox{if $\xi$ is small and $\xi \perp \rho$.}
\end{array}
\right.
$$
In particular, if $\kappa =0$, then
the function $m(\cdot+i\rho)$ satisfies a standard
Mihlin--H\"ormander condition of order $J$ at infinity on $\fra^*$
and a nonisotropic Mihlin--H\"ormander condition of order $J$
near the origin. A similar anisotropy was noticed
in \cite[Thm~1~\rmvii\ and \rmix]{CGM1} in connection with the kernel of
the (modified) Poisson semigroup.
\end{remark}

In the next proposition we prove that
if $M \in \fH(\bP;J,\kappa)$, then the restriction of 
$M\circ Q$ to $\TW$ belongs to $H'(\TW;J,\kappa)$.
A straightforward calculation then implies that if 
$m$ is  holomorphic and bounded in $\TW$ and satisfies
estimates (\ref{f: pseudodiff estimates}),
then the product $m\, (M\circ Q)$ is in $H'(\TW;J,\kappa)$.

\begin{proposition} \label{p: McircQ}
Suppose that $J$ is a nonnegative integer, and that $\kappa$
is in $[0,\infty)$.
Then there exists a constant $C$ such that 
$$
\norm{M\circ Q}{H'(\TW;J,\kappa)}
\leq C \, \norm{M}{\fH(\bP;J,\kappa)}
\quant M \in \fH(\bP;J,\kappa).
$$
\end{proposition}

\begin{proof}
By arguing as in the proof of 
Proposition~\ref{p: McircQ0}, 
we see that there exists a constant~$C$ such~that 
\begin{equation}\label{f: zetabig}
\mod{D^{(I',i_\ell)} (M\circ Q)(\zeta)}
\leq C\, \norm{M}{\fH(\bP;J,\kappa)} \,
\bigmod{Q(\zeta)}^{-(i_\ell+\mod{I'})/2}
\quant \zeta \in \bigl(\fra^*\setminus \Ball{\phantom{}}\bigr)+i\bW^+.
\end{equation}

We claim that there exists a constant $C$ such that
\begin{equation}\label{f: claim}
\mod{\zeta'}
\leq C \, \mod{Q(\zeta)}^{1/2}
\quant \zeta \in \Ball{\phantom{}}+i\bW^+.
\end{equation}

Given the claim, we indicate how to conclude the proof
of the proposition.
Write $I$ for the multiindex $(I',i_\ell)$.
Note that (\ref{f: induction}), the assumption $M\in \fH(\bP;J,\kappa)$
and (\ref{f: claim}) imply~that there exists a constant
$C$ such~that 
$$
\begin{aligned}
\mod{D^{(I',i_\ell)} (M\circ Q)(\zeta)}
& \leq C\, \norm{M}{\fH(\bP;J,\kappa)} \,
           \sum_{0\leq P\leq I/2} \mod{\zeta_1}^{(i_\ell-2p_\ell)}\,
           \mod{\zeta'}^{\mod{I'}-2\mod{P'}}\,
           \bigmod{Q(\zeta)}^{-\kappa-i_\ell-\mod{I'}+\mod{P}} \\
& \leq C\, \norm{M}{\fH(\bP;J,\kappa)} \,\sum_{0\leq P\leq I/2}
           \bigmod{Q(\zeta)}^{ \mod{I'}/2-\mod{P'} }\,
           \bigmod{Q(\zeta)}^{-\kappa-i_\ell-\mod{I'}+\mod{P}}\\
& \leq C\,\norm{M}{\fH(\bP;J,\kappa)} \,\bigmod{Q(\zeta)}^{-\kappa-
           i_\ell-\mod{I'}/2}
           \quant \zeta \in \Ball{\phantom{}}+i\bW^+. 
\end{aligned}
$$
The required conclusion follows directly fr{}om this estimate
and (\ref{f: zetabig}).

It remains to prove the claim.
We abuse the notation and denote by $\Cone{c_1}$
the cone $\{(\la',\la_\ell) \in \fra^*: \mod{\la'} < c_1\, \la_\ell\}$.
By \cite[Lemma~8.3]{H1} 
$\bW^+ = (\fra^*)^+ \cap \bigl(\rho - {}^+(\fra^*)\bigr)$,
where ${}^+(\fra^*)$ denotes the dual cone of $(\fra^*)^+$.
Recall that ${}^+ (\fra^*) \subset \Cone{c_1}$ (see (\ref{f: inclusion
between cones})), so that 
$\bW^+ \subset (\fra^*)^+ \cap \bigl(\rho -\Cone{c_1}\bigr)$.
Suppose that $c$ is a number such that $(c_1^2-1)/(c_1^2+1) < c < 1$.
Set
$
\bV = \{(\eta',\eta_\ell): c \,\mod{\rho}< \eta_\ell<\mod{\rho},\,
\mod{\eta'} < c_1\,(\mod{\rho}-\eta_1) \}.
$
If $c$ is sufficiently close to $1$, then $\bV \subset (\fra^*)^+$.

Observe that (\ref{f: claim}) is obvious when
$\zeta$ is in $\Ball{\phantom{}} + i (\bW^+ \setminus \bV)$. 
Indeed, both sides of (\ref{f: claim})
are continuous functions of $\zeta$, and $\zeta$
stays at a positive distance from $i\rho$, which is
the unique point in $\OTWpiu$ where $Q$ vanishes.

Now suppose that $\zeta$ is in $\Ball{\phantom{}} + i (\bW^+ \cap \bV)$,
and write $\zeta = \xi + i \eta$.
Note that
$$
\begin{aligned}
\mod{Q(\zeta)}^2
& = (\mod{\xi}^2+\mod{\rho}^2-\mod{\eta}^2)^2
        +4\mod{\prodo{\xi}{\eta}}^2\\
& \geq (\mod{\xi}^2+\mod{\rho}^2-\mod{\eta}^2)^2.
\end{aligned}
$$
Furthermore
$$
\begin{aligned}
\mod{\rho}^2-\mod{\eta}^2
& =\mod{\rho}^2-\eta_1^2-\mod{\eta'}^2 \\
& \geq \mod{\rho}^2-\eta_1^2-c_1^2(\mod{\rho}-\eta_1)^2\\
& = \bigl(\mod{\rho}-\eta_1 \bigr)\,
        \bigl(\mod{\rho}+\eta_1-c_1^2\mod{\rho}+c_1^2\eta_1 \bigr)\\
& \geq \mod{\rho} \, \bigl(\mod{\rho}-\eta_1 \bigr)\,
        \bigl[1-c_1^2+ c \, (1+c_1^2) \bigr].
\end{aligned}
$$
Since $c$ has been chosen so that $1-c_1^2+ c \, (1+c_1^2) >1$, 
$$
\mod{\rho}^2-\mod{\eta}^2
\geq (\mod{\rho}-\eta_1)\,\mod{\rho}
\geq \smallfrac{\mod{\rho}}{c_1}\,\mod{\eta'}.
$$
Therefore
$$
\mod{Q(\zeta)}^2
 \geq C\,(\mod{\xi}^2+\mod{\eta'})^2
 \geq C\,(\mod{\xi'}^4+\mod{\eta'}^2)
 \geq C\,\mod{\zeta'}^4.
$$
This completes the proof of the claim (\ref{f: claim}), 
and of the proposition.
\end{proof}

Now we state our main result.  Its proof is deferred to Section~\ref{s:
Proof}.
Given $B$ in $\GOp{2}$, we denote by
$\opnorm{B}{1;1,\infty}$ the quasi-norm of $B$
\emph{qua} operator from $\lu{X}$ to $\lorentz{1}{\infty}{X}$.

\begin{theorem} \label{t: main result}
Denote by $J$ the integer $[\![n/2]\!]+1$.
The following hold:
\begin{enumerate}
\item[\itemno1]
if $\kappa$ is in $[0,1)$, then there exists a constant $C$ such that
for all $B$ in $\GOp{2}$ for which $m_B$ is in $H'(\TW;J,\kappa)$
$$
\opnorm{B}{1;1,\infty}
\leq C \, \norm{m_B}{H'(\TW;J,\kappa)};
$$
\item[\itemno2]
there exists a constant $C$ such that 
$$
\opnorm{M(\cL)}{1;1,\infty}
\leq C \, \norm{M}{\fH(\bP;J,1)}
\quant M \in \fH(\bP;J,1).
$$
\end{enumerate}
\end{theorem}

\begin{remark}
The proof of Theorem~\ref{t: main result}
will show that in the case where $\ell>1$ the nonisotropic
behaviour of the multiplier $m_B$ near the point $i\rho$
(see Remark~\ref{rem: nonisotropic cond} above)
implies a nonisotropic behaviour of
the kernel $k_B$ at infinity. In fact, the bounds of
$k_B$ we shall obtain are expressed, in Cartan co-ordinates, 
in terms of a nonisotropic homogeneous ``norm'' on $\fra$.
\end{remark}

\begin{remark}
Observe that Theorem~\ref{t: main result}~\rmii\ 
applies to $\cL^{-\al/2}$ when $0 \leq \Re \al \leq 2$
(hence we re-obtain Anker's result \cite{A2}),
and that it is sharp, in
the sense that for each $\kappa >1$
the function $z\mapsto z^{-k}$ is in $\fH(\bP;J,\kappa)$,
but $\cL^{-\kappa}$ is not of weak type $1$.
We also remark that if $M$ is in $\fH(\bP;J,\kappa)$ for 
some $\kappa$ in $[0,1)$, then, \emph{a fortiori},
$M$ is in $\fH(\bP;J,1)$, hence \rmii\ applies to~$M$.
\end{remark}

\begin{remark}
We do not know whether \rmi\ holds with $\kappa =1$. 
Moreover, if $M$ is in $\fH(\bP;J,1)$, then $m_{M(\cL)}$
is in $H'(\TW;J,1)$ by Proposition~\ref{p: McircQ}. Thus, 
for functions of the Laplace--Beltrami operator $\cL$
 condition \rmi\ is weaker than \rmii.
\end{remark}

\section{Weak type estimates for certain convolution operators}
\label{s: The higher rank case I}

Suppose that $\vep\in \BR$, and
consider the $K$--bi-invariant functions 
$\tau_1^\vep$ and $\tau_2^\vep$ on $G$, defined~by
\begin{equation}\label{f: estimatetau}
\begin{aligned}
\tau_1^\vep \bigl(\exp H \bigr)
& = \e^{-\rho(H)- \mod{\rho}\, \mod{H}} \,\big(1+\rho(H)\big)^{(1-\ell)/2
      -\vep} \\
\tau_2^\vep \bigl(\exp H \bigr)
&= \e^{-2\rho(H)} \,\big(1+\cN(H)\big)^{1-\ell-\vep}
\end{aligned}
\quant H \in \fra^+.
\end{equation}
The homogeneous norm $\cN$ is defined in (\ref{f: anisotropic norm}).
Note that 
%the functions $\tau_1^\vep$ and
%$\tau_2^\vep$ are bounded. Moreover,
$\tau_1^\vep \notin \lu{G}$ when $\vep\leq (\ell+1)/2$, and
$\tau_2^\vep \notin \lu{G}$ when $\vep\leq 2$.
We denote by $T_1^\vep$ and $T_2^\vep$
the convolution operators 
$f\mapsto f*\tau_1^\vep$ and $f\mapsto f*\tau_2^\vep$ respectively.
In this section we study the weak type $1$ boundedness
of the operators $T_1^\vep$ and $T_2^\vep$.
The weak type $1$ estimate for $T_1^0$
was essentially proved by Str\"omberg in \cite{Str} (see also
\cite[pag. 1331]{AL} and \cite[pag. 276]{A2}).

It is fair to say that 
the result stated in \cite[Remark~2, p. 125]{Str} applies to
both $\tau_1^\vep$ when $\vep \geq 0$ and to $\tau_2^\vep$ 
when $\vep>0$. 
This gives the weak type $1$ estimate for 
$T_1^\vep$ when $\vep \geq 0$ and for $T_2^\vep$ 
when $\vep>0$. 
However, the result in \cite[Remark~2, p. 125]{Str} 
is stated without proof.   For the
reader's convenience we prefer to give a self-contained proof
of the weak type $1$ estimate for the operator $T_2^\vep$.
Our strategy follows closely that of Str\"omberg,

For each complex number $b$ denote by $\e_b$ the character
$
s\mapsto \e^{bs}
$
on $\BR$.
Recall the co-ordinates $(H',H_\ell)$ on $\fra$
introduced in Section~\ref{s: Notation}.
Denote by $\nu$ the measure on $\fra$ defined by
$\wrt\nu (H',H_\ell)=\e^{2|\rho|H_\ell}\wrt H_\ell\wrt H'$.
Note that $\nu$
is the product measure $\la_{\ell-1}\times \nu_1$, where
$\la_{\ell-1}$ denotes the Lebesgue measure on $\rho^\perp$ and
${\wrt\nu_1} (H_\ell) = \e^{2|\rho|H_\ell}\wrt H_\ell$.
Define the function $\si$ by
\begin{equation}\label{sigma}
\si(H',H_\ell)
= \e_{-2\mod{\rho}}(H_\ell)\, p(H')
\quant(H',H_\ell)\in\BR^{\ell-1}\times \BR\,.
\end{equation}
where $p$ is a function in $\lu{\la_{\ell-1}}$. 
Define the operators $S_1$ and $S$ by
$$
S_1 f
= f\ast_{\BR} \e_{-2|\rho|}
\quant f\in C^{\infty}_c(\BR)
\qquad \hbox{and}\qquad
S f
= f \ast_{\BR^{\ell}}\si
\quant f \in C^{\infty}_c(\BR^{\ell})
$$
where $\ast_{\BR}$ and $\ast_{\BR^\ell}$ denote the convolution on
$\BR$ and on $\BR^\ell$ respectively.
Observe that
\begin{equation} \label{f: composition}
\begin{aligned}
S f(H',H_\ell)
& = \int_{\BR}\e_{-2|\rho|}(H_\ell-L_\ell)\,
     \int_{\BR^{\ell-1}}p(H'-L')\,f(L',L_\ell) \wrt L' \wrt L_\ell\\
& = \int_{\BR}\e_{-2|\rho|}(H_\ell-L_\ell)\,
     \big[f\ast_{\BR^{\ell-1}} p(\cdot,L_\ell) \big](H') \wrt L_\ell\\
& = \big[S_1F(H',\cdot) \big](H_\ell)
     \quant (H',H_\ell)\in\BR^{\ell-1}\times \BR,
\end{aligned}
\end{equation}
where $F(H',\cdot)(H_\ell)
= \big[f\ast_{\BR^{\ell-1}} p(\cdot,H_\ell) \big](H')$.
Note that
\begin{equation} \label{f: partial estimate}
\norm{F(H',\cdot)}{\lu{\nu_1}}
\leq \int_{\BR^{\ell-1}} \norm{f(L',\cdot)}{\lu{\nu_1}}
\, \mod{p(H'-L')}\wrt L'.
\end{equation}

We shall use the following elementary lemma.

\begin{lemma}\label{l: weak 1dim}
The following hold:
\begin{enumerate}
\item[\itemno1]
the operator $S_1$ extends to a bounded operator from $\lu{\nu_1}$
to $\lorentz{1}{\infty}{\nu_1}$;
\item[\itemno2]
the operator $S$ extends to a bounded operator from
$\lu{\nu}$ to $\lorentz{1}{\infty}{\nu}$.
\end{enumerate}
\end{lemma}

\begin{proof}
First we prove \rmi.
It suffices to consider nonnegative functions $f$. Since $\e_{-2|\rho|}$ is
a character of the group $\BR$,
$$
S_1f
= \e_{-2|\rho|}\, \big[(\e_{2|\rho|}f) \ast_{\BR} {\bf{1}} \big],
$$
where ${\bf{1}}$ denotes the constant function equal $1$ on $\BR$. Observe
that
$$
(\e_{2|\rho|}f)\ast_{\BR}{\bf{1}}(s)
= \norm{\e_{2|\rho|}f}{\lu{\la_1}}
= \norm{f}{\lu{\nu_1}}.
$$
Thus,
$
S_1f
=\e_{-2|\rho|}\,\|f\|_{\lu{\nu_1}}\,.
$
Now, for every $t>0$ the level set $\{s\in\BR:~S_1f(s)>t\}$ is just the interval
$
\bigl(-\infty,\log\bigl(\norm{f}{\lu{\nu_1}}/t \bigr)^{1/2|\rho|} \bigr).
$
Hence
\begin{align*}
\nu_1\bigl(\{s\in\BR: S_1f(s) > t \} \bigr)
& = \int_{-\infty}^{\log(\norm{f}{\lu{\nu_1}}/t)^{1/(2\mod{\rho})}}
     \wrt \nu_1\\
& = \frac{1}{2|\rho|}\,\frac{\|f\|_{\lu{\nu_1}}}{t}
    \quant t \in \BR^+,
\end{align*}
as required.

Now we prove \rmii. 
Suppose that $f$ is in $\lu{\nu}$.
By Fubini's theorem, (\ref{f: composition}) and (\ref{f: partial estimate})
\begin{align*}
\nu\big( \{ H \in \fra: \mod{S f (H)} > t \} \big)
& = \int_{\rho^\perp} \nu_1\big( \{ (H_\ell \in\BR:
      \mod{S f (H',H_\ell)} > t \} \big) \wrt H'\\
& = \int_{\rho^\perp} \nu_1\big( \{ H_\ell \in\BR:
      \bigmod{[S_1F(H',\cdot) \big](H_\ell)} >t \} \big) \wrt H'\\
& \leq \frac{1}{2\mod{\rho}t}\,
      \int_{\rho^\perp}\|F(H',\cdot) \|_{\lu{\nu_1}} \wrt H'\\
& \leq \frac{1}{2\mod{\rho}t}\,\|p\|_{\lu{\la_{\ell-1}}}\,
      \|f\|_{\lu{\nu}} \quant t \in \BR^+,
\end{align*}
as required.
\end{proof}

For each $c$ in $\BR^+$ define the cone $\Cone{c}$ by 
\begin{equation} \label{f: Cone} 
\Cone{c}
= \{ H\in \fra: \mod{H'} < c\, H_\ell \}.
\end{equation}
Since $H_{\rho}$ is in $\fra^+$,
there exists $c_0$ such that $\Cone{c_0} \subset \fra^+$.
It is well known (see \cite[Lemma~34]{HC}
or \cite[Ch. VII, Lemma~2.20~\rmiv]{H2})
that the dual Weyl chamber ${}^+\fra$ contains $\fra^+$.
Then the dual cone $\Cone{1/c_0}$ contains ${}^+\fra$. Choose
$c_1 > 1/c_0$: note that
\begin{equation} \label{f: inclusion between cones}
\Cone{c_0} \subset \fra^+ \subset {}^+\fra \subset \Cone{1/c_0}
\subset \Cone{c_1}.
\end{equation}

\begin{proposition} \label{p: debole per tau}
Suppose that $\vep$ is in $\BR$. The following hold:
\begin{enumerate}
\item[\itemno1]
the operator $T_1^\vep$ is of weak type $1$ if and only if $\vep \geq 0$;
\item[\itemno2]
if $\ell>1$, then the operator $T_2^\vep$ is of weak type $1$ 
if and only if $\vep > 0$.  If $\ell = 1$, then 
$T_2^\vep$ is of weak type $1$ 
if and only if $\vep \geq 0$.
\end{enumerate}
\end{proposition}

\begin{proof}
First we prove \rmi.  Str\"omberg \cite{Str}
proved the weak type $1$ boundedness of the convolution operator
$f\mapsto f*\tau$, where $\tau$ is the $K$--bi-invariant
function defined by 
$$
\tau\bigl(\exp H \bigr)
= \e^{-2\mod{\rho}\, \mod{H}} \,\big(1+\mod{H}\big)^{(1-\ell)/2}
\quant H \in \fra^+.
$$
It is straightforward to check that his argument applies
almost \emph{verbatim} to the operator $T_1^0$.
Since $\tau_1^\vep \leq \tau_1^0$ for all $\vep>0$,
the weak type $1$ estimate for the operators $T_1^\vep$ is
an immediate consequence of that of $T_1^0$.

To conclude the proof of \rmi\ it remains to show that
$T_1^\vep$ is not of weak type $1$ when $\vep<0$.
By a standard argument, it suffices to prove that
the corresponding kernel $\tau_1^\vep$ is
not in $\lorentz{1}{\infty}{X}$.  We give the details in the case where
$\ell \geq 2$.  Those in the case where $\ell=1$ are easier, and are omitted.
Observe that
$$
\tau_1^\vep (\exp H)
= \e^{-2\rho(H)} \, U(H)
\quant H \in \fra^+,
$$
where
$
U(H) = \e^{\rho(H) - \mod{\rho}\, \mod{H}}\, \bigl(1+\rho(H)
\bigr)^{(1-\ell)/2-\vep}.
$
Write $H = (H',H_\ell)$, and recall that $\rho(H) = \mod{\rho} H_\ell$. 
A straightforward computation shows that
$$
\begin{aligned}
\rho(H) - \mod{\rho}\, \mod{H}
& = - \mod{\rho} \, \smallfrac{\mod{H'}^2}{H_\ell+\sqrt{H_\ell^2+\mod{H'}^2}} \\
& \geq - \mod{\rho} \, \smallfrac{\mod{H'}^2}{H_\ell}.
\end{aligned}
$$
Now, if $H$ is in 
$\Parabolic$ (see (\ref{f: Parabolic})), then $\mod{H'}^2/H_\ell \leq 1$,
so that there exists a positive constant $c$ such that
$$
U(H)
\geq c\, \bigl(1+H_\ell \bigr)^{(1-\ell)/2-\vep}
\quant H \in \Parabolic.
$$
For each $t$ in $\BR^+$, define
$
E_t = \{ k_1 \exp(H) k_2\in K \exp( \fra^+) K:
\tau_1^\vep(\exp H) >t\}.
$
Set
$
h := \inf\{ H_\ell \in \BR^+: (H',H_\ell) \in \fra^+\cap \Parabolic \},
$
and denote by $s_t$ the unique point in $\BR^+$ such that
\begin{equation} \label{f: unique point}
\smallfrac{\e^{2\mod{\rho}s_t}}{(1+s_t)^{(1-\ell)/2-\vep}}
= (ct)^{-1}.
\end{equation}
Denote by $\mod{E_t}$ the Haar measure of $E_t$.  Note that
$$
\begin{aligned}
\mod{E_t}
& \geq \Bigmod{\Bigl\{k_1 \exp(H',H_\ell)\, k_2 
     \in K\exp(\fra^+ \cap \Parabolic)\,K:
     \smallfrac{\e^{-2\mod{\rho}\, H_\ell}}{c\,
     \bigl(1+H_\ell \bigr)^{(\ell-1)/2+\vep}} > t \Bigr\}} \\
& \geq \bigmod{\bigl\{k_1 \exp(H',H_\ell)\, k_2 
     \in K\exp(\fra^+ \cap \Parabolic)\,K:
     h < H_\ell < s_t \bigr\}}.
\end{aligned}
$$
It is straightforward to check that
this measure is estimated from below by a constant times
$
\int_h^{s_t} s^{(\ell-1)/2} \, \e^{2\mod{\rho}s} \wrt s.
$
By (\ref{f: unique point})
$s_t$ tends to $\infty$ as $t$ tends to $0^+$.
Integration by parts shows that the integral above is
comparable to $s_t^{(\ell-1)/2} \, \e^{2\mod{\rho}s_t}$ as $t$
tends to $0^+$. Thus, there exists a positive constant~$C$ such that
$$
\mod{E_t}
\geq C \,s_t^{(\ell-1)/2} \, \e^{2\mod{\rho}s_t} 
\geq C \, \smallfrac{s_t^{-\vep}}{t}.
$$
Hence $\sup_{t>0} \bigl( t\mod{E_t}\bigr) = \infty$,
so that $\tau_1^\vep\notin\lorentz{1}{\infty}{X}$ if $\vep<0$.
The proof of~\rmi\ is complete.

Next we prove \rmii.   
Suppose first that $\ell = 1$.  Then 
$$
\begin{aligned}
\tau_1^\vep (\exp H) 
& = \e^{-2\rho(H)} \, (1+ \mod{\rho} \, H)^{-\vep}
       \\
\tau_2^\vep (\exp H) 
& = \e^{-2\rho(H)} \, (1+ \sqrt{H})^{-\vep}
\end{aligned}
\quant H \in \fra^+. 
$$
It is straightforward to check that there exist positive constants
$C_1$ and $C_2$ such that 
$$
C_1 \, \tau_1^{\vep/2} 
\leq \tau_2^\vep
\leq C_2 \, \tau_1^{\vep/2}. 
$$
By \rmi\ $T_1^{\vep/2}$ is of weak type $1$ if and only if $\vep \geq 0$. 
Hence so is $T_2^\vep$, as required.

Now suppose that $\ell \geq 2$ and that $\vep > 0$.
We express $\tau_2^\vep$ in Iwasawa co-ordinates.
Denote by $P: \OV{N} \to \BR$ the function defined by
$
P(\OV{n})
= \e^{-\rho(H(\OV{n}))}.
$
Recall that
\begin{equation} \label{f: Cartan component}
[\OV{n}ak]^+
= \log a+H(\OV{n})+H'\bigl(\OV{n}, a \bigr)
\quant \OV{n} \in \OV{N} \quant a\in A \quant k \in K,
\end{equation}
where $H(\OV{n})$ and $H'\big(\OV{n},a \big)$
are in ${}^+\fra$
(see, for instance, \cite[p.~119]{Str}),
and $[\OV{n}ak]^+$ denotes the~$\OV{\fra^+}$ component of $\OV{n}ak$ in
the Cartan decomposition $K \exp{\OV{\fra^+}} K$.
For the rest of the proof we write $x$ instead of $\log a$,
and $y$ instead of $H(\OV{n})+H'(\OV{n},a)$.
Then
\begin{equation}\label{f: estimatetau1}
\tau_2^\vep(\OV{n}ak)
= \e^{-2\rho(x+y)} \,
      \bigl[1+\cN\bigl(x+y\bigr)\bigr]^{1-\ell-\vep}.
\end{equation}
Since $H'\big(\OV{n},a \big)$ is in ${}^+\fra$,
$
\e^{-\rho(y)}
\leq \e^{-\rho(H(\OV{n}))} = P(\OV{n}),
$
so that
\begin{equation}\label{f: e-2rz}
\tau_2^{\vep}(\OV{n}ak)
\leq \e^{-2\rho(x)} \, P(\OV{n})^{3/2}
         \, \e^{-\rho(y)/2} \,
      \bigl[1+\cN\bigl(x+y\bigr)\bigr]^{1-\ell-\vep}.
\end{equation}
We claim that there exists a positive constant $C$ such that
\begin{equation}\label{f: pointwise estimate}
\frac{\e^{-\rho(y)/2}}{\bigl[1+\cN\bigl(x+y\bigr)\bigr]^{\ell+\vep-1}}
\leq
\begin{cases}
C\, \bigl[ 1+\cN(x) \bigr]^{1-\ell-\vep}
& \quant x\in \Cone{c_1} \\
\e^{-\mod{\rho}\, \mod{x_\ell}/2}
& \quant x\in (-\Cone{c_1}) \\
\e^{-\de \mod{x'}-\de c_1 \mod{x_\ell}}
& \quant x\notin \bigl(\Cone{c_1} \cup (-\Cone{c_1})\bigr),
\end{cases}
\end{equation}
where $\de = (c_0-1/c_1)\, \mod{\rho}/8$
(see (\ref{f: Cone}) and (\ref{f: inclusion between cones})
for the definitions of $c_0$, $c_1$ and $\Cone{c_1}$).

Observe that, on the one hand, $x+y$ belongs to $\OV{\fra^+}$, because
$x+y = [\OV{n}ak]^+$, hence to $\Cone{c_1}$.
On the other hand $y$ is in ${}^+\fra\subset \Cone{c_1}$, so that
that $x+y$ is in $x+\Cone{c_1}$. Thus,
$$
x+y \in \Cone{c_1} \cap \bigl(x+\Cone{c_1}\bigr).
$$
To prove the claim, first assume that $x$ is in $\Cone{c_1}$.
Observe that if $\cN(y)\leq \cN(x)/2$, then
$$
\cN(x)\leq \cN(x+y)+\cN(-y)\leq \cN(x+y)+\cN(x)/2.
$$
Hence $\cN(x)\leq 2\, \cN(x+y)$ and
$\cN(x+y)^{1-\ell-\vep } \leq C \,\cN(x)^{1-\ell-\vep}$, so that
$$
\frac{\e^{-\rho(y)/2}}{\bigl[1+\cN\bigl(x+y\bigr)\bigr]^{\ell+\vep-1}}
\leq C \, \bigl[1+\cN\bigl(x\bigr)\bigr]^{1-\ell-\vep},
$$
where we have used the fact that $\rho(y) \geq 0$.

If, instead, $\cN(y)> \cN(x)/2$,
we observe that
$
\cN(x+y) \geq (x_\ell+y_\ell)^{1/2}
$
by definition of the homogeneous norm $\cN$, and that
$
\cN(y) \leq (1+c_1^4)^{1/4} \, \mod{y_\ell},
$
because $y$ is in $\Cone{c_1}$, and conclude that
$$
%\begin{aligned}
\frac{\cN(x)}{\cN(x+y)}
\leq \frac{2\, \cN(y)}{\sqrt{x_\ell+y_\ell}}
\leq 2 \, (1+c_1)^{1/4}\, \sqrt{y_\ell}.
%\end{aligned}
$$
In the last inequality we have also used the fact that $x_\ell>0$,
because $x$ is in the cone $\Cone{c_1}$.
Then $\cN(x+y)^{1-\ell-\vep}\leq C\, y_\ell^{(\ell+\vep-1)/2} \,
\cN(x)^{1-\ell-\vep}$. Hence
$$
\begin{aligned}
\frac{\e^{-\rho(y)/2}}{\bigl[1+\cN\bigl(x+y\bigr)\bigr]^{\ell+\vep-1}}
& \leq C \, [1+y_\ell^{(\ell+\vep-1)/2}] \,\e^{-\rho(y)/2}\,
     \cN(x)^{1-\ell-\vep} \\
& \leq C \, \cN(x)^{1-\ell-\vep},
\end{aligned}
$$
as required.

Next suppose that $x$ is in $-\Cone{c_1}$. Since $x+y$ is the
$\OV{\fra^+}$ component of $\OV n a k$
in the Cartan decomposition $K\exp(\OV{\fra^+})K$,
$x+y$ is in $\Cone{c_1}$, hence $x_\ell+y_\ell\geq 0$. Therefore
$y_\ell\geq -x_\ell$. Now $-x_\ell=\mod{x_\ell}$, because
$x$ is in $-\Cone{c_1}$. Hence
$$
\frac{\e^{-\rho(y)/2}}{\bigl[1+\cN\bigl(x+y\bigr)\bigr]^{\ell+\vep-1}}
\leq \e^{\mod{\rho}x_\ell/2}\,
= \e^{-\mod{\rho}\, \mod{x_\ell}/2},
$$
as required.

Finally, suppose that $x$ is in
$\fra\setminus \bigl(\Cone{c_1}\cup (-\Cone{c_1}) \bigr)$.
Since $x+y$ is in $\Cone{1/c_0}$, $y_\ell+x_\ell>c_0\, \mod{x'+y'}$. Hence
$$
y_\ell
 > -x_\ell+ c_0 \, \mod{x'+y'}
\geq -x_\ell+ c_0 \, \bigl(\mod{x'}-\mod{y'}\bigr).
$$
Recall that $y$ is in $\Cone{1/c_0}$, whence $-c_0\,\mod{y'} > - y_\ell$,
and that $x$ is in $\fra\setminus \bigl(\Cone{c_1}\cup (-\Cone{c_1}) \bigr)$,
so that $-\mod{x_\ell} > -\mod{x'}/c_1$. Therefore
$$
y_\ell
\geq \bigl( c_0-1/c_1\bigr)\, \mod{x'} -y_\ell
\geq \frac{c_0-1/c_1}{2}\, \mod{x'}
       + \frac{c_0\, c_1-1}{2}\, \mod{x_\ell} -y_\ell,
$$
i.e., $y_\ell > \bigl(c_0-1/c_1\bigr)\, \mod{x'}/4 +
\bigl(c_0\, c_1-1\bigr)\, \mod{x_\ell}/4$. Hence
$$
\frac{\e^{-\rho(y)/2}}{\bigl[1+\cN\bigl(x+y\bigr)\bigr]^{\ell+\vep-1}}
\leq \e^{-(c_0-1/c_1)\,\mod{\rho}\, \mod{x'}/8}
\, \e^{-(c_0\, c_1-1)\,\mod{\rho}\, \mod{x_\ell}/8},
$$
as required to conclude the proof of the claim.

Denote by $\si_2^\vep$ the function defined by
\begin{equation}\label{f: pointwise estimate II}
\si_2^{\vep}(\exp x)
=
\begin{cases}
C\, \e^{-2\rho(x)} \, \bigl( 1+\cN(x) \bigr)^{1-\ell-\vep}
& \quant x\in \Cone{c_1} \\
\e^{-2\rho(x)-\mod{\rho(x)}/2}
& \quant x\in (-\Cone{c_1}) \\
\e^{-2\rho(x)-\de \mod{x - \rho(x)\rho/\mod{\rho}}
-\de c_1 \mod{\rho(x)}/\mod{\rho}}
& \quant x\notin \bigl(\Cone{c_1} \cup (-\Cone{c_1})\bigr).
\end{cases}
\end{equation}
It is straightforward to check that
$\si_2^\vep$ is in $\lu{\fra\setminus \Cone{c_1}, \nu}$.
Hence the corresponding convolution operator 
is of weak type $1$.

Note that
$
\cN(x) \geq \mod{x'}.
$
Fr{}om (\ref{f: pointwise estimate II}) we deduce that
$$
\bigl(\si_2^\vep\One_{\Cone{c_1}}\bigr)(\exp x) 
\leq C \, \frac{\e^{-2\mod{\rho}\, x_\ell}}{(1+\mod{x'})^{\ell +\vep-1}}
\quant x \in \Cone{c_1}.
$$
Since $x'\mapsto (1+\mod{x'})^{\ell +\vep-1}$
is in $\lu{\rho^\perp,\la_{\ell-1}}$ for all $\vep > 0$, we may apply
Lemma~\ref{l: weak 1dim} and conclude that the operator
$f\mapsto f*(\si_2^\vep\circ \log )$ is bounded from $\lu{\fra,\nu}$
into $\lorentz{1}{\infty}{\fra,\nu}$.

Now, (\ref{f: e-2rz}) and (\ref{f: pointwise estimate}) imply that
\begin{equation}\label{f: pointwise estimate III}
\tau_2^{\vep}(\OV{n}ak)
\leq P(\OV{n})^{3/2} \, \si_2^\vep(a).
\end{equation}
It is well known (see, for instance, \cite{Str})
that $P^{3/2}$ is in $\lu{\OV{N}}$.
This, estimate (\ref{f: pointwise estimate III}) and the fact that
$f\mapsto f*(\si_2^\vep\circ \log )$ is bounded from $\lu{\BR^\ell,\nu}$
into $\lorentz{1}{\infty}{\BR^\ell,\nu}$ imply
(see \cite[Step four, p.~118--120]{Str})
that the map $f \mapsto f*\tau_2^\vep$ is of weak type~$1$, as required.

To conclude the proof of \rmii, it remains
to show that $T_2^0$ is not of weak type $1$.
It suffices to prove that $\tau_2^0$ is not in $\lorentz{1}{\infty}{X}$.
Denote by $\tau'$ the $K$--bi-invariant function
on $G$ defined~by
$$
\tau'\bigl(k_1\exp( H',H_\ell)k_2 \bigr)
= (1+\mod{H'})^{1-\ell} \, \e^{-2\mod{\rho}\, H_\ell}
\, \One_{\Parabolic^c\cap \Cone{c_0}} (H',H_\ell)
\quant H \in \fra^+ \quant k_1,k_2 \in K.
$$
Note that
$
\tau_2^0\bigl(\exp H \bigr)
\geq C \, \tau'(\exp H)
$
for all  $H$ in $\fra^+$.
Indeed, $H_\ell\leq \mod{H'}^2$, because $H$ is in $\Parabolic^c$. 
Hence $1+\cN(H) \leq 1+ 2^{1/4}\, \mod{H'}$, fr{}om which the
inequality above follows directly.

We show that $\tau'$ is not in $\lorentz{1}{\infty}{X}$.  
Clearly this implies
that $\tau_2^0$ is not in $\lorentz{1}{\infty}{X}$ either, as required.
For each $t$ in 
$(0,\e^{-2\mod{\rho}} \, 2^{1-\ell}]$ define $\Om_t =
\bigl\{k_1\exp( H)\, k_2 \in K\exp(\fra^+)\, K : \tau' (k_1\exp( H) k_2) 
> t\bigr\},
$
and the function $b_t: \BR\to \BR$ by 
$$
b_t(s)
= (t\, \e^{2\mod{\rho}s})^{-1/(\ell-1)}-1 \quant s\in \BR.
$$
Denote by $u_t$ and $v_t$ the unique solutions to the equations
$s = b_t(s)$ and $s^{1/2} = b_t(s)$.
It is straightforward to check that $1 < u_t <v_t$ for all
$t$ in $(0,\e^{-2\mod{\rho}} \, 2^{1-\ell}]$ and that
$s^{1/2} < b_t(s) <s$ for all $s$ in $(u_t,v_t)$.
Note also that $\tau'(\exp H) > t$ if and only if $H$ is in
$\Parabolic^c\cap \Cone{c_0}$ and $\mod{H'} < b_t(H_\ell)$.
Therefore 
$$
\Om_t 
\supset \bigl\{ k_1\exp(H',H_\ell)\, k_2\in K\exp( \fra^+)\, K:  
u_t < H_\ell < v_t, \, H_\ell^{1/2} < \mod{H'} < b_t(H_\ell) \bigr\},
$$
and
$$
\mod{\Om_t}
\geq \int_{u_t}^{v_t} \e^{2\mod{\rho}s} \, \la_{\ell-1}(A_s){\wrt s},
$$
where $A_s$ denotes the annulus
$\{H'\in \rho^\perp: s^{1/2} < \mod{H'}
< b_t(s) \}$.
Therefore
$$
\la_{\ell-1} (A_s) 
= c\,b_t(s) - c \, s^{(\ell-1)/2},
$$
where $c$ is the volume of the unit ball in $\BR^{\ell-1}$
with respect to the Lebesgue measure.
Observe that $u_t$ tends to $\infty$ as $t$ tends to $0^+$. Hence
$s$ is large in the formula above.
Now, there exists a positive constant $C$ such that
if $s$ is large, then
$$
\la_{\ell-1} (A_s) 
\geq \frac{C}{t} \, \e^{-2\mod{\rho}s},
$$
so that
$$
\mod{\Om_t}
\geq C \, \frac{v_t - u_t}{t}
\quant t \in (0,\e^{-2\mod{\rho}} \, 2^{1-\ell}].
$$
To conclude the proof, it suffices to show that $v_t - u_t$
does not stay bounded as $t$ tends to~$0^+$. Fr{}om
the definition of $u_t$ and $v_t$ we deduce that
$$
\e^{2\mod{\rho}(v_t - u_t)}
= \Bigl(\frac{1+v_t}{1+u_t^{1/2}}\Bigr)^{\ell-1}.
$$
Now, if $v_t-u_t$ stays bounded, then so does the
right hand side in the formula above. Hence there exists
a constant $C$ such that $1+v_t \leq C \, (1+u_t^{1/2})$, but
this is impossible, because $v_t > u_t$ and $u_t$ tends to $\infty$
as $t$ tends to $0^+$.

This proves that $T_2^0$ is not of weak type $1$, as required
to conclude the proof of \rmii\ and of the proposition.
\end{proof}

\section{Kernel estimates} \label{s: Kernel estimates}

In this section we prove some technical lemmata,
which will be used in the proof of Theorem~\ref{t: main result}.
The ball $\Ball{\phantom{}}$ is defined just below formula
(\ref{f: Ball}).

\begin{lemma}\label{l: powers of Q}
Suppose that $\ga$ is in $\BR^+$.  Then there exists
a constant $C$ such that for every $\eta$ in
$(\fra^*)^+$ with $\mod{\eta}=\mod{\rho}$ and for every
$\vep$ in $(0,1/4)$
\begin{equation}\label{f: integralQ}
\int_{\Ball{\phantom{}}}\bigmod{Q\bigl(\la+i(1-\vep)\eta\bigr)}^{-\ga}\wrt\la
\leq
\begin{cases}
C \,\bigl(1+\vep^{(\ell+1)/2-\ga}\bigr) & \hbox{if $\ga \neq (\ell+1)/2$} \\
C\, \log(1/\vep)                        & \hbox{if $\ga = (\ell+1)/2$}.
\end{cases}
\end{equation}
\end{lemma}

\begin{proof}
Given $\eta$ in $(\fra^*)^+$ such that $\mod{\eta}=\mod{\rho}$,
we choose an orthonormal basis of $\fra^*$ whose last vector 
is $\eta/\mod{\eta}$.  For any $\la$ in $\fra^*$ we write
$\la=(\lae',\lae)$, where $\lae'\in \BR^{\ell-1}$ and $\lae \in \BR$
for the co-ordinates of $\la$ with respect to this orthonormal basis. 
Notice that
$$
\bigmod{ Q \bigl(\la + i (1- \vep) \eta \bigr) }^{2}
= \bigl[ \lae^2 + \mod{\lae'}^2 +(2\vep-\vep^2)\mod{\rho}^2 \bigr]^2
     + 4 \, (1-\vep)^2 \, \mod{\rho}^2\lae^2.
$$
Then there exists a constant $C$ such that
$$
\bigmod{ Q \bigl(\la + i (1- \vep) \eta \bigr) }^{2}
\geq C \, \bigl[( \mod{\lae'}^2 + \vep )^2 + \lae^2\bigr] 
\quant \la \in \Ball{\phantom{}}.
$$
Therefore
\begin{equation}  \label{f: integral} 
\int_{\Ball{\phantom{}}} \bigmod{ Q \bigl(\la + i (1- \vep) \eta \bigr) }^{-\ga}
\wrt\la 
\leq C\, \int_{\Ball{\phantom{}}}\frac{1}{\bigl[ 
( \mod{\lae'}^2 + \vep )^2 + \lae^2 \bigr]^{\ga/2} }\wrt\lae \wrt \lae'.
\end{equation}
If $\ell+1 > 2\ga$, then the integral on the right hand side
of (\ref{f: integral}) is estimated from above by 
$$
\int_{\Ball{\phantom{}}}  \bigl[ \mod{\lae'}^4 + \lae^2 \bigr]^{-\ga/2}
\wrt \lae \wrt \lae',
$$
which is finite, so that (\ref{f: integralQ})
is proved in this case.

Now suppose that $\ell+1 \leq 2\ga$.
We abuse the notation and denote by $\NBall{R}$ the set of
all $(\lae',\lae)$ in $\BR^{\ell-1}\times \BR$ such that 
$\mod{\lae'}^4 + \lae^2 < R^4$.  Observe that $\Ball{\phantom{}}\subset
\NBall{2\mod{\rho}}$.  Indeed, if $(\lae',\lae)$ is in~$\Ball{\phantom{}}$, then
$\mod{\lae'}^2+\lae^2 < \mod{\rho}^2$.  In particular 
$\mod{\lae'}< \mod{\rho}$ and $\mod{\lae}< \mod{\rho}$,
whence 
$$
\mod{\lae'}^4+\lae^2 < \mod{\rho}^2\, \mod{\lae'}^2+\lae^2  
< \max (1,\mod{\rho}^2)\, (\mod{\lae'}^2+\lae^2)< \max (1,\mod{\rho}^4)
\leq (2\mod{\rho})^4,
$$
because $\mod{\rho}$ is always at least $1/2$.
We majorise the
integral on the right hand side of (\ref{f: integral}) by
integrating on $\NBall{2\mod{\rho}}$ instead than on $\Ball{\phantom{}}$.  Then,
changing variables
$(\lae',\lae) = (\vep^{1/2}\, v', \vep \, v_\ell)$,
we see that 
$$
\int_{\Ball{\phantom{}}} \bigmod{ Q \bigl(\la + i (1- \vep) \eta \bigr) }^{-\ga}
\wrt\la 
\leq C  \int_{\NBall{4\mod{\rho}/{\sqrt{\vep}}}}
\frac{ \vep^{(\ell+1)/2-\ga} }{ \bigl[(\mod{v'}^2+1)^2+v_\ell^2\bigr]^{\ga/2}}\wrt v' \wrt v_\ell.
$$

If $\ell+1=2\ga$, then (\ref{f: integral}) is bounded by $C\,\log(1/\vep)$,
as required.

If $\ell+1<2\ga$, then (\ref{f: integral}) is bounded by
$$
\vep^{(\ell+1)/2-\ga} \Bigl[\int_{\NBall{1}}
\frac{\wrt v' \wrt v_\ell}{ \bigl[(\mod{v'}^2+1)^2+v_\ell^2\bigr]^{\ga/2}}
+ \int_{\NBall{1}^c} \frac{\wrt v' 
\wrt v_\ell}{(\mod{v'}^4 + v_\ell^2)^{\ga/2}} \Bigr] 
\leq C \, \vep^{(\ell+1)/2-\ga},
$$
as required.
\end{proof}

Lemma~\ref{l: pointwiseestimates} below will 
be used in {Step II} of the proof of 
Theorem~\ref{t: main result} to control the kernel~$k_B$ 
away from the walls of $\fra^+$, whereas Lemma~\ref{l: inductivestep}
below is needed in {Step III} of the same proof
to control the size of $k_B$ near the walls of $\fra^+$.

\begin{lemma}\label{l: pointwiseestimates}
Suppose that $\kappa$ is in $[0,1]$.  Set
$J:= \ell+1$ and denote by $L$ the least integer $\geq (\ell+1)/2$.
For any function $m$, which is holomorphic in $\TWpiut$,
for some $t$ in $\BR^-$, and such that $m \, \, \e^{-Q/2}$
is in $\lu{\fra^*}$ (with respect to the 
Lebesgue measure), define $k_1: \fra^+ \to \BC$~by
$$
k_1 (H)
= \int_{\fra^*} m(\la) \, \e^{-Q(\la)/2}\, \e^{i\la(H)} \wrt \la
\quant H \in \fra^+.
$$
The following hold:
\begin{enumerate}
\item[\itemno1]
there exists a constant $C$ such that for all $m$ in $H'(\TWpiut;J,\kappa)$
and for all $H$ in $\fra^+$
$$
\mod{k_1(H)}
\leq 
\begin{cases}
C\, \norm{m}{H'(\TWpiut;J,\kappa)} \,
{\e^{-\rho(H)}}{\bigl[ 1+ \cN(H)\bigr]^{-\ell-1+2\kappa} }
 & \hbox{if $0<\kappa\leq 1$} \\ 
C\, \norm{m}{H'(\TWpiut;J,0)} \,
     \e^{-\rho(H)}\, {\bigl[ 1+ \cN(H)\bigr]^{-\ell-1} } \, 
    \log \bigl[ 2+ \cN(H)\bigr] 
&  \hbox{if $\kappa = 0$}; 
\end{cases}
$$
\item[\itemno2]
if either $0<\kappa\leq 1$ or $\kappa = 0$ and $\ell$ is even, then
there exists a constant $C$ such that for all $m$ in $H'(\TBpiut;J,\kappa)$
$$
\mod{k_1(H)}
\leq 
C\, \norm{m}{H(\TBpiut;L,\kappa)} \,
{\e^{-\mod{\rho}\,\mod{H}}}{\bigl[1+ 
\rho(H)\bigr]^{-(\ell+1)/2+\kappa}}
\quant H \in \fra^+.
$$
Similarly, if $\kappa=0$ and $\ell$ is odd, then 
there exists a constant $C$ such that for all $m$ in $H'(\TBpiut;J,\kappa)$
$$
\mod{k_1(H)}
\leq 
C\, \norm{m}{H(\TBpiut;L,0)} \,
\e^{-\mod{\rho}\,\mod{H}}\, \bigl[1+ \rho(H)\bigr]^{-(\ell+1)/2}
\, \log \bigl[2+ \rho(H)\bigr]
\quant H \in \fra^+.
$$
\end{enumerate}
\end{lemma}

\begin{proof}
We denote by $m_1$ the function defined by
$$
m_1 (\zeta)= m (\zeta)\, \, \e^{-Q(\zeta)/2}
\quant  \zeta \in\TWpiut .
$$
Observe that $k_1$ (which is the inverse Fourier transform of 
$m_1$) is bounded, because
$m_1$ is in $\lu{\fra^*}$.  Therefore all the estimates 
in \rmi\ and \rmii\ hold trivially for $H$ in 
$\fra^+\cap\NBall{2}$, and we may assume
that $H$ is in $\fra^+\cap \NBall{2}^c$.

First we prove (i).  For the duration of the proof of \rmi\
we write $\rho_\vep$ instead of $(1-\vep)\rho$.
An application of Leibniz's r\^ule shows that there
exists a constant $C$ such that for every
multiindex $(I',i_\ell)$ such that $\mod{I'}+i_\ell \leq J$,
for every $\vep$ in $(0,1/4)$, and  
for every $m$ in $H'(\TWpiut;J,\kappa)$
\begin{equation} \label{f: estimate for derivatives}
\begin{aligned}
& \mod{ D^{(I',i_\ell)} m_1(\la+i\rho_\vep)} \\
&\leq 
\begin{cases}
C \, \norm{m}{H'(\TWpiut;J,\kappa)}
\, {\e^{-\Re Q(\la+i\rho_\vep)/4}}\, {\bigmod{Q\bigl(\la+i\rho_\vep
\bigr)}^{-(i_\ell+\mod{I'})/2}} 
  & \quant \la \in \fra^*\setminus\Ball{\phantom{}} \\     
C \, \norm{m}{H'(\TWpiut;J,\kappa)}
\, {\bigmod{Q\bigl(\la+i\rho_\vep
\bigr)}^{-\kappa-i_\ell-\mod{I'}/2}} 
  & \quant \la \in \Ball{\phantom{}}.
\end{cases}
\end{aligned}
\end{equation}
Assume that $\vep$ is in the interval $\bigl(0,C/\rho(H)\bigr)$ for 
some fixed constant $C$.
Since $m_1$ is holomorphic in $\TWpiut$, we may move the
contour of integration to the space $\fra^*+i \rho_\vep$,
and obtain 
$$
\begin{aligned}
\mod{k_1(H)}
& =  \e^{- (1-\vep)\rho(H)} \, 
      \Bigmod{\int_{\fra^*} m_1\bigl(\la+i\rho_\vep\bigr)
      \, \e^{i\la(H)} \wrt \la}  \\
& \leq C \, \e^{-\rho(H)} \Bigmod{\int_{\fra^*} m_1\bigl(\la+i\rho_\vep\bigr)
      \, \e^{i\la(H)} \wrt \la}
\quant H \in \fra^+\cap \NBall{2}^c.
\end{aligned}
$$
We shall treat the cases where $H$ is in
$\fra^+\cap \NBall{2}^c\cap \Parabolic$ and $H$ is in
$\fra^+\cap\NBall{2}^c\cap\Parabolic^c$ separately (the region $\Parabolic$
is defined in (\ref{f: Parabolic})).

First suppose that $H$ is in $\fra^+\cap \NBall{2}^c\cap \Parabolic$
and choose $\vep = 1/\rho(H)$.
By integrating by parts~$J$ times
with respect to the variable $\la_\ell$, we see that
$$
\begin{aligned}
\mod{k_1(H)}
& \leq C \, \e^{-\rho(H)} \Bigmod{i^{-J} \, H_\ell^{-J}
       \int_{\fra^*} m_1(\la+i\rho_\vep) \,
       \, \partial_\ell^J \e^{i\la(H)} \wrt \la} \\
& =    C \, \e^{-\rho(H)} \Bigmod{ (-i)^{-J} \, H_\ell^{-J}
       \int_{\fra^*} \partial_\ell^J m_1(\la+i\rho_\vep) \,
       \e^{i\la(H)} \wrt \la} \\
& \leq \smallfrac{C}{H_\ell^J} \, \e^{-\rho(H)}
    \Bigl[\int_{\fra^*\setminus\Ball{\phantom{}}}
          \mod{ \partial_\ell^J m_1(\la+i\rho_\vep)} \wrt \la
   + \int_{\Ball{\phantom{}}} \mod{ \partial_\ell^J m_1(\la+i\rho_\vep)}
          \wrt \la \Bigr].
\end{aligned}
$$
We use estimates (\ref{f: estimate for derivatives}) with
$I'=0'$ and $i_\ell = J$, and obtain 
$$
\mod{k_1(H)}
\leq C\, \norm{m}{H'(\TWpiut;J,\kappa)}\, 
 \smallfrac{\e^{-\rho(H)}}{H_\ell^J} \,
 \Bigl[\int_{\fra^*\setminus \Ball{\phantom{}}} 
 \frac{\e^{-\Re Q(\la+i\rho_\vep)/4}}{\bigmod{Q\bigl(\la+i\rho_\vep
 \bigr)}^{i_\ell/2}} \wrt \la 
  + \int_{\Ball{\phantom{}}} \bigmod{ Q(\la+i\rho_\vep) }^{-\kappa-J}\wrt \la\Bigr].
$$
It is straightforward to check that $\Re Q(\la+i\rho_\vep) \geq \mod{\la}^2$
for all $\la$ in $\fra^*$.  Hence the first integral is majorised by
$
\int_{\fra^*\setminus \Ball{\phantom{}}} 
\exp (-\mod{\la}^2/4)\, \mod{\la}^{-i_\ell} \wrt \la, 
$
which is clearly convergent and independent of $\vep$.  To estimate
the second integral we observe that $\kappa + J > (\ell+1)/2$
for every $\kappa$ in $[0,1]$.  Then Lemma~\ref{l: powers of Q}
(with $\ga = \kappa + J$) implies that 
$$
\int_{\Ball{\phantom{}}} \mod{ Q(\la+i\rho_\vep) }^{-\kappa-J}\wrt \la
\leq C \, (1+ \vep^{(\ell+1)/2-J-\kappa}) \quant \vep \in (0,1/4).
$$  
Recall that $\vep = 1/\rho(H)$,
and that $H$ is in $\fra^+\cap\NBall{2}^c\cap\Parabolic$, so that
$H_\ell$ is (positive and) bounded away from $0$. Therefore
$$
\begin{aligned}
\mod{k_1(H)}
& \leq C\, \norm{m}{H'(\TWpiut;J,\kappa)}\,\smallfrac{\e^{-\rho(H)}}{H_\ell^J}
    \, \bigl[1 + H_\ell^{J+\kappa-(\ell+1)/2} \bigr] \\
& \leq C \, \norm{m}{H'(\TWpiut;J,\kappa)}\,\, \e^{-\rho(H)} \, H_\ell^{\kappa-(\ell+1)/2} \\
& \leq C \, \norm{m}{H'(\TWpiut;J,\kappa)}\,\, \e^{-\rho(H)} \, \bigl[ 1+ \cN(H) \bigr]^{2\kappa-(\ell+1)}
\quant H \in \fra^+\cap\NBall{2}^c\cap\Parabolic,
\end{aligned}
$$
as required.

Next suppose that $H$ is in $\fra^+\cap\NBall{2}^c\cap\Parabolic^c$ 
and choose $\vep = 1/\mod{H'}^2$.
Note that $\vep \leq C/\rho(H)$, where $C$ does not depend on $H$.
Suppose that $H = (H',H_\ell)$ is given.
Denote by $\partial'$ the directional derivative on $\fra^*$
in the direction of $H'$.
By integrating by parts, we see that
\begin{equation} \label{f: secondo caso}
\begin{aligned}
\mod{k_1(H)}
& \leq C \, \e^{-\rho(H)} \Bigmod{i^{-J} \, \mod{H'}^{-J}
       \int_{\fra^*} m_1 \bigl(\la+i\rho_\vep\bigr) \,
       (\partial')^J \e^{i\la(H)} \wrt \la} \\
& =    C \, \e^{-\rho(H)} \Bigmod{\mod{H'}^{-J}
       \int_{\fra^*} (\partial')^J m_1 \bigl(\la+i\rho_\vep\bigr) \,
       \e^{i\la(H)} \wrt \la}.
\end{aligned}
\end{equation}
%where $\Delta'$ denotes the standard Laplacian in the variables $\la'$.
By arguing much as above (we use (\ref{f: estimate for derivatives})
with $\mod{I'} = J$ and $i_\ell = 0$), we see that if $\kappa >0$,
then
$$
\begin{aligned}
\mod{k_1 (H)}
& \leq C\,\norm{m}{H'(\TWpiut;J,\kappa)}\,
        \smallfrac{\e^{-\rho(H)}}{\mod{H'}^{J}} \, \,
        \bigl[1 + \vep^{(\ell+1-J)/2-\kappa} \bigr] \\
& \leq C\,\norm{m}{H'(\TWpiut;J,\kappa)}\,
         \smallfrac{\e^{-\rho(H)}}{\mod{H'}^{J}} \, \,
         \bigl[1 + \mod{H'}^{J+2\kappa-(\ell+1)} \bigr] \\
& \leq C \,\norm{m}{H'(\TWpiut;J,\kappa)}\,
         \e^{-\rho(H)} \, \bigl[1 + \cN(H) \bigr]^{2\kappa-(\ell+1)}
\quant H \in \fra^+\cap\NBall{2}^c\cap\Parabolic^c,
\end{aligned}
$$
as required to conclude the proof of \rmi\ in the case $\kappa>0$.
If, instead, $\kappa = 0$, then by arguing much as above we see
that 
$$
\mod{k_1(H)}
\leq C\, \norm{m}{H'(\TWpiut;J,0)}\, 
 \smallfrac{\e^{-\rho(H)}}{\mod{H'}^J} \,
 \Bigl[\int_{\fra^*\setminus \Ball{\phantom{}}} 
 \frac{\e^{-\Re Q(\la+i\rho_\vep)/4}}{\bigmod{Q\bigl(\la+i\rho_\vep
 \bigr)}^{J}} \wrt \la 
  + \int_{\Ball{\phantom{}}} \bigmod{ Q(\la+i\rho_\vep) }^{-J}\wrt \la\Bigr].
$$
By Lemma~\ref{l: powers of Q} the last integral is estimated 
by $C \, \log(1/\vep)$, so that 
$$
\begin{aligned}
\mod{k_1(H)}
& \leq C\, \norm{m}{H'(\TWpiut;J,0)}\, 
     \smallfrac{\e^{-\rho(H)}}{\mod{H'}^J} \, \log \mod{H'} \\
&  \leq C\, \norm{m}{H'(\TWpiut;J,0)} \,
     {\e^{-\rho(H)}\, \log \bigl[ 2+ 
     \cN(H)\bigr] }{\bigl[ 1+ \cN(H)\bigr]^{-\ell-1} },
\end{aligned}
$$
where we have used the fact that there exists a positive constant $c$
such that 
$$
c \,\, \cN(H',H_\ell) \leq \mod{H'} \leq \cN(H',H_\ell)
\quant (H',H_\ell) \in 
\fra^+\cap\NBall{2}^c\cap\Parabolic^c.
$$
The proof of \rmi\ is complete. 

Next we prove (ii).
Observe that for any vector $\eta$ in $\partial \bB^+$ and any positive
integer $j\leq L$ the derivative $\partial_\eta^j m$ of order
$j$ in the direction of $\eta$ may be written as a linear combination
of the derivatives $D^I m$ with $\mod{I}=j$.
Therefore
\begin{equation}\label{f: derivativeeta}
\mod{\partial_{\eta}^jm(\zeta)}\leq C\, \norm{m}{H(\TBpiut;J,\kappa)} \,
\mod{Q(\zeta)}^{-\kappa-j}
\quant \zeta\in \Ball{\phantom{}}+i\bB^+.
\end{equation}
By the Leibniz r\^ule, $m_1$ satisfies a similar estimate.
Given $H$ in $\fra^+\cap\NBall{2}^c$, define $\vep$ and $\eta$ by
$
\vep= {1}/({\mod{\rho}\,\mod{H}})$
and $\eta=({\mod{\rho}}/{\mod{H}})\,H$.
For the duration of the proof of \rmii\
we write $\eta_\vep$ instead of $(1-\vep)\eta$.
By shifting the integration to the space
$\fra^*+i \eta_\vep$, and integrating by parts
$L$ times, we see that
$$
\begin{aligned}
k_1(H)
& = \e^{- (1-\vep)\mod{\rho} \mod{H}} \int_{\fra^*}
       m_1\bigl(\la+i\eta_\vep \bigr)
       \, \e^{i\la(H)} \wrt \la \\
& = \smallfrac{\e^{- (1-\vep)\mod{\rho} \mod{H}}}{(i\, \eta(H))^{L}}
       \int_{\fra^*} m_1\bigl(\la+i\eta_\vep \bigr)
       \, \partial_{\eta}^L\e^{i\la(H)} \wrt \la \\
& = \smallfrac{\e^{- (1-\vep)\mod{\rho} \mod{H}}}{(-i\, \mod{\rho}\,
      \mod{H})^{L}}
       \int_{\fra^*} \partial_{\eta}^L m_1\bigl(\la+i\eta_\vep \bigr)
       \, \e^{i\la(H)} \wrt \la.
\end{aligned}
$$
By arguing as in the proof of \rmi\
we see that there exists a constant $C$ such that for 
every $\vep$ in $(0,1/4)$
\begin{equation} \label{f: outside b1II}
\int_{\fra^*\setminus\Ball{\phantom{}}}
\mod{ \partial_\eta^L m_1(\la+i\eta_\vep)} \wrt \la
\leq C \, \norm{m}{H(\TBpiut;L,\kappa)}
      \quant m \in H(\TBpiut;L,\kappa).
\end{equation}
This and (\ref{f: derivativeeta}) imply that
$$
\mod{k_1(H)}
\leq C \, \norm{m}{H(\TBpiut;L,\kappa)}\, \smallfrac{\e^{-\mod{\rho} 
   \mod{H} }}{\mod{H}^L} \, \Bigl[1 +
\int_{\Ball{\phantom{}}} \bigmod{Q\bigl(\la+i\eta_\vep \bigr) }^{-\kappa-L}
\wrt \la \Bigr].
$$
We use Lemma~\ref{l: powers of Q} to estimate the last integral.
If $\kappa = 0$ and $\ell$ is odd, then $L = (\ell+1)/2$.
Therefore the last integral is majorised
by $C \, \log(1/\vep)$.  Thus,
$$
\begin{aligned}
\mod{k_1(H)}
& \leq C \, \norm{m}{H(\TBpiut;L,0)}\, \smallfrac{\e^{-\mod{\rho} 
        \mod{H} }}{\mod{H}^L} \,\log(1/\vep) \\
& \leq C\, \norm{m}{H(\TBpiut;L,0)} \,
        {\e^{-\mod{\rho}\,\mod{H}}} \, {\bigl[1+ \rho(H)\bigr]^{-(\ell+1)/2}}
        \,  \log \bigl[2+ \rho(H)\bigr],
\end{aligned}
$$
where we have used the fact that if $H$ is in $\fra^+$, then
$\rho(H)=\mod{\rho}H_1\leq \mod{\rho}\mod{H}$.
If, instead, either $\ell$ is even, or $\kappa >0$, then
$L+\kappa > (\ell+1)/2$, so that by Lemma~\ref{l: powers of Q}
$$
\begin{aligned}
\mod{k_1(H)}
& \leq C \, \norm{m}{H(\TBpiut;L,\kappa)}\,
       \smallfrac{\e^{-\mod{\rho} \mod{H} }}{\mod{H}^L} \,
       \bigl[ 1 + \mod{H}^{L+\kappa-(\ell+1)/2} \bigr]\\
& \leq C \, \norm{m}{H(\TBpiut;L,\kappa)}\,
       \e^{-\mod{\rho} \mod{H} } \,
       \bigl[ 1+ \mod{H} \bigr]^{\kappa-(\ell+1)/2}\\
& \leq C \, \norm{m}{H(\TBpiut;L,\kappa)}\,
       \e^{-\mod{\rho} \mod{H} } \,
       \bigl[ 1+ \rho(H) \bigr]^{\kappa-(\ell+1)/2}
\quant H \in \fra^+\cap\NBall{2}^c.
\end{aligned}
$$
The proof of \rmii\ is complete.
\end{proof}

\begin{definition}
For any $s$ in $[0,\infty)$ define the function $\Upsilon^s$
and the measure~$\muD$~by
$$
\Upsilon^s (\la) 
= (1+ \mod{\la})^s
\qquad\hbox{\and}\qquad
\wrt \muD(\la) = \Upsilon^s(\la) \, \wrt \la
\quant \la \in \fra^*.
$$
Suppose that $\bE$ is a Weyl invariant subset of\, $\bW$, and that
$J$ is a nonnegative integer.  
Denote by $\Sspace{\bE,J}$ the vector space of all Weyl invariant 
holomorphic functions $m$ in $\TE$ such that~$S_{\bE,J}^{s} (m) < \infty$
for all $s$ in $[0,\infty)$,~where 
$$
S_{\bE,J}^{s} (m)
= \max_{\mod{I} \leq J} \, \sup_{\eta\in \bE}\, 
\int_{\fra^*}\mod{D^I m(\la+i\eta)} \wrt \muD(\la).
$$
We endow $\Sspace{\bE,J}$ with the locally convex topology
induced by the family of seminorms $\{S_{\bE,J}^{s} :~s\in [0,\infty) \}$.
With this topology $\Sspace{\bE,J}$ becomes a Fr\'echet space.
\end{definition}

\begin{remark} \label{r: est for Upsilon}
Observe that for every $s$ in $[0,\infty)$
there exists a constant $C$ such that 
$$
\Upsilon^s (\la) 
\leq (1+ \mod{\la+ i \eta})^s
\leq C \, \Upsilon^s (\la) 
\quant \la \in \fra^* \quant \eta \in \bW.
$$
Consequently
$$
S_{\bE,J}^{s}(m)
\leq \max_{\mod{I} \leq J} \, \sup_{\eta\in \bE}\, 
\int_{\fra^*}\mod{D^I m(\la+i\eta)} \, (1+ \mod{\la+ i \eta})^s \wrt \la
\leq C \, S_{\bE,J}^{s} (m)
\quant m \in \Sspace{\bE,J}.
$$
We shall use this observation without any further comment.
\end{remark}

For any nontrivial subset $F$ of $\Sppp$ and $0<\de\leq\vep<\infty$
define the region $\pareteF{F}{\de}{\vep}$~by
\begin{equation} \label{f: cono-parete}
\pareteF{F}{\de}{\vep} 
=\{ H\in\Strip_2: \hbox{$\al(H)\leq \de\mod{H}$
$\,\,\forall \al\in F$, and $\al(H)\geq \vep\mod{H}$ 
$\,\,\forall \al\in \Sppp\setminus F$}\}.
\end{equation}
In the following proposition we put together some useful
facts concerning the sets $\pareteF{F}{\de}{\vep}$ that 
will be used below.
For any $c$ in $\BR^+$ define $(\Strip_F)_c$ and $(\Strip^F)_c$ by
$$
(\Strip_F)_c = \{H_F \in \OV{\aFb}: 0\leq \om_F(H_F) \leq c\}
\qquad\hbox{and}\qquad
(\Strip^F)_c = \{H^F \in \OV{\aFa}: 0\leq \om^F(H^F) \leq c\}.
$$

\begin{lemma} \label{l: insiemi}
Suppose that $F$ is a nontrivial subset of $\Sppp$.
The following hold:
\begin{enumerate}
\item[\itemno1]
if $H$ is in $\pareteF{F}{\de}{\vep} \cap \NBall{1}^c$, 
then 
$\HFb$ is in $\OV{(\aFb)^+}$, $\HFa$ is in $(\aFa)^+$,
\begin{equation}\label{f: estimatesinconeF}
\omFa(\HFa)\geq \omFa(H)\geq \vep\mod{H}\geq \vep
\qquad {\rm{and}} \qquad \mod{\HFb}\leq \gamma\,\de\mod{H},
\end{equation}
where $\gamma$ is a positive constant which depends 
on the root system $\Sigma$;
\item[\itemno2]
if $H$ is in $\pareteF{F}{\de}{\vep}$,
then $H_F$ is in $(\Strip_F)_2$.
\end{enumerate}
\end{lemma}

\begin{proof}
For the proof of \rmi\ see \cite[3.16.2-3.16.4]{AJ}.  

To prove \rmii\ suppose that $H$ is in $\pareteF{F}{\de}{\vep}$ and
that $\om(H) = \al(H)$ for some $\al$ in $\Sppp$.
If~$\al$ is in $F$, then $\al(H_F) = \al(H) \leq 2$.
If, instead, $\al$ is in $\Sppp\setminus F$, then 
$$
\al(H) \geq \vep \, \mod{H} \geq \smallfrac{\vep}{\de}\, \be(H)
\quant \be \in F.
$$
Hence
$$
\om_F(H_F) 
\leq \smallfrac{\de}{\vep}\, \al(H) 
\leq 2\, \smallfrac{\de}{\vep} \leq 2,
$$
so that $H_F$ is in $(\Strip_F)_2$, as required.

\end{proof}

Define $\si$ by
\begin{equation}\label{f: sigma}
\si
= \min\{\mod{\rFb}: \emptyset\subset F\subseteq \Sppp\},
\end{equation}
and denote by $\Es{\si}$ the Weyl invariant subset of $\bW$ defined by
$$
\Es{\si}
= \{\eta\in\bW: \mod{\eta - w\cdot \rho}\geq\si \, \hbox{for all $w\in W$}\}.
$$
Set $\Cosh_{2\rho}(H) := \sum_{w\in W} \e^{2w\cdot \rho(H)}$ for all $H$
in $\fra$ and denote by $\cM_{2\rho}$ the multiplication operator 
acting on $K$--bi-invariant functions $f$ on $G$ by
$$
\bigl(\cM_{2\rho} f\bigr) (\exp H)
=  \Cosh_{2\rho} (H)  \, f(\exp H)
\quant H\in \fra.
$$
Note that there exist positive constants $C_1$ and $C_2$ such that
\begin{equation} \label{f: stima Cosh}
C_1 \, \e^{2\rho(H)}
\leq \Cosh_{2\rho}(H)
\leq C_2 \, \e^{2\rho(H)}
\quant H \in \OV{\fra^+}.
\end{equation}

The proof of the following lemma is reminiscent of 
the proof of \cite[Thm~3.7]{AJ}
and of that of the main result in \cite[Section~7.10]{GV}.  
All these proofs use the Trombi--Varadarajan
expansion of spherical functions and an induction argument.

\begin{lemma}\label{l: inductivestep}
The following hold:
\begin{enumerate}
\item[\itemno1]
the map $\cM_{2\rho} \circ \cH^{-1}$ is bounded from 
$\Sspace{\bE_\si,0}$ to $\ly{\Strip_2}$;
\item[\itemno2]
if $J \geq \ell + 1$, then
the map $\cM_{2\rho} \circ \cH^{-1}$ is bounded from
$\Sspace{\bE_\si,J}$ to $\lu{\Strip_2}$ (with respect
to the Lebesgue measure).
\end{enumerate}
\end{lemma}

\begin{proof}
Suppose that $m$ is in $\Sspace{\Esi,J}$, and 
denote by $k$ its inverse spherical Fourier transform
$$
k(\exp H)=
\int_{\fra^*} \vp_{\la}(\exp H)\,m(\la) \planl
\quant H\in\fra.
$$
It is straightforward to check that this integral is
absolutely convergent. 

First suppose that $\ell=1$. 
Then $\Strip_2$ is the interval $\{H\in \OV{\fra^+}: 0\leq \al(H)
\leq 2\}$, where $\al$ denotes the unique simple positive root.
In particular, $\Strip_2$ is a bounded subset of $\OV{\fra^+}$, and the function
$H\mapsto \e^{2\rho(H)}$ is bounded on $\Strip_2$.
Furthermore, $\si = \mod{\rho}$, so that $\bE_{\si} = \{0\}$.
Now, (\ref{f: estimatec}) and the fact that $\norm{\vp_{\la}}{\infty} = 1$ 
for any $\la$ in $\fra^*$ imply~that
$$
\begin{aligned}
\mod{\e^{2\rho(H)}k(\exp H)}
& \leq C\, \int_{\fra^*} \mod{m(\la)}
        \, \bigl(1+\mod{\la}\bigr)^{n-\ell} \wrt\la\\
& =    C\,  \Ss{n-\ell}{0}(m)
    \quant H \in \Strip_2,
\end{aligned}
$$  
where $C$ does not depend on $m$ in $\Sspace{\Esi,J}$.
Therefore, by (\ref{f: stima Cosh}),
$$
\norm{\cM_{2\rho}k}{\ly{\Strip_2}}
\leq C\,\Ss{n-\ell}{0}(m), 
$$
whence 
$$
\norm{\cM_{2\rho}k}{\lu{\Strip_2}}
\leq C\,\Ss{n-\ell}{0}(m), 
$$
because $\Strip_2$ has finite measure. 
This proves both \rmi\ and \rmii\ in the case where $\ell = 1$.
 
Now suppose that $\ell \geq 2$, and that $m$ is in $\Sspace{\bE_{\si},J}$.
We observe preliminarily that, arguing as we did above in the case
where $\ell =1$, we may show that
%
%$0$ is in $\Esi$ and that 
%$\norm{\vp_{\la}}{\infty} = 1$ for any $\la$ in $\fra^*$.  Thus,
%$$
%\begin{aligned}
%\mod{\e^{2\rho(H)}k(\exp H)}
%& \leq C\, \int_{\fra^*} \mod{m(\la)}
%        \, \bigl(1+\mod{\la}\bigr)^{n-\ell} \wrt\la\\
%& =    C\,  \Ss{n-\ell}{0}(m)
%    \quant H \in \Strip_2 \cap \NBall{1},
%\end{aligned}
%$$  
%where $C$ does not depend on $m$ in $\Sspace{\Esi,J}$.
%Therefore, by (\ref{f: stima Cosh}),
$$
\norm{\cM_{2\rho}k}{\ly{\Strip_2\cap \NBall{1}}}
\leq C\,\Ss{n-\ell}{0}(m). 
$$
Since $\Strip_2\cap \NBall{1}$ has finite measure,
\begin{equation} \label{f: near the origin}
\norm{\cM_{2\rho}k}{\lu{\Strip_2\cap \NBall{1}}}
\leq C\,\Ss{n-\ell}{0}(m).
\end{equation}

Thus, in the rest of the proof we may assume that $H\in \bs_2
\setminus \NBall{1}$.

A consequence of \cite[Lemma 2.1.7]{AJ}
is that $\Strip_2$ is covered by a finite number of 
regions $\pareteF{F}{\de_F}{\vep_F}$,
where $\emptyset \subset F\subseteq \Sppp$,
$\de_F$ and $\vep_F$ may be chosen so that 
$0<\de_F\leq \vep_F<\infty$, and $\de_F$ is as small as we~need.
We shall prove that $\cM_{2\rho}k$ is either bounded or
integrable in $\Strip_2$ 
by showing that $\cM_{2\rho}k$ is bounded or integrable 
respectively in 
$\pareteF{F}{\de_F}{\vep_F}$ for every nontrivial subset $F$ of $\Sppp$. 

Fix $F \subseteq \Sppp$, $\de_F$ and $\vep_F$ as above.
By using the Trombi--Varadarajan asymptotic expansion for the
spherical functions,
and the Weyl invariance of $m$, for each positive integer~$N$ we may write 
$$
k(\exp H)
= \sum_{q\in \QFa,\,\height{q}<N} h_q^F (H) + r_N^F  (H)
\quant H \in \pareteF{F}{\de_F}{\vep_F},
$$
where $h_q^F(H)$ is defined,
for every $H$ in $\pareteF{F}{\de_F}{\vep_F}$, by
\begin{equation}\label{f: hq}
h_q^F(H)
=  \mod{W_F \backslash W}\,
     \e^{-\rFa(H)}\,\int_{\fra^*}\mod{\bcFb(\la)}^{-2}\,
     \bigl[(\vbcFa)^{-1}m\bigr](\la)\,\phiFa{\la}{q}(\exp H) \wrt\la
\end{equation}
and $r_N^F$ is a remainder term. 
We extend $h_q^F$ and $r_N^F$ to $\Strip_2$ by setting them equal to $0$
outside~$\pareteF{F}{\de_F}{\vep_F}$.

First we prove \rmi. 
We argue by induction on the rank $\ell$ of the symmetric space. 
We have already proved \rmi\ in the case where $\ell=1$.
Suppose that \rmi\ holds for all symmetric spaces of the 
noncompact type and rank $\leq \ell-1$, 
and consider a symmetric space $X$ of the noncompact type and rank~$\ell$.  

Consider the remainder term $r_N^F$. 
By Theorem~\ref{t: asymptoticexpansion}~(iv) and (\ref{f: estimatesinconeF}) 
there exist positive constants $C$ and $d$ such that
\begin{equation}\label{f: pointwiserN}
\begin{aligned}
\bigmod{r_N^F(H)}
&  \leq C\, \e^{-\rho(H)-N\omFa(H)}\,\bigl( 1+\mod{H} \bigr)^{d} \,
    \int_{\fra^*}\mod{m(\la)}\,\bigl( 1+\mod{\la} \bigr)^{d}\wrt\la\\
&  \leq C\,  \e^{-2\rho(H)}\,\e^{\mod{\rho}\mod{H}-N\vep_F\mod{H}}\,
    \bigl( 1+\mod{H} \bigr)^{d}\, \Ss{d}{0}(m) 
    \quant H \in \pareteF{F}{\de_F}{\vep_F}.
\end{aligned}
\end{equation}
Choose $N > \mod{\rho}/\vep_F$.  Then 
\begin{equation} \label{f: norm rN}
\norm{\cM_{2\rho}r_N^F}{\ly{\Strip_2}}
\leq C\,\Ss{d}{0}(m).
\end{equation}

Next, suppose that $q$ is in $\QFa\setminus \{0\}$ with $\height{q}<N$. 
We may write the integral in (\ref{f: hq}) as an iterated integral,
where the outer integral is on $\astFb$ and the inner integral on $\astFa$.

For the rest of the proof for each $v\in(0,1)$ we shall write
$\rFa_v$ instead of $(1-v) \, \rFa$.

Since $m$ is holomorphic in $\TW$, $\phiFa{\la}{q}$ and 
$(\vbcFa)^{-1}$ are holomorphic 
in a neighborhood of $T_{(\astFa)^+}$, for each $v\in(0,1)$
we may move the contour of integration in the inner integral 
to the space $\astFa+i \rFa_v$, and obtain
$$
h_q^F (H) 
= \mod{W_F\backslash W} \,
    \e^{-\rFa(H)}\,\int_{\astFb} \bigmod{\bcFb(\laFb)}^{-2}\, m_q(\laFb) \wrt
    \laFb \quant H \in \pareteF{F}{\de_F}{\vep_F},
$$
where
$$
m_q(\laFb)=
\int_{\astFa} \bigl[ (\vbcFa)^{-1} \, m \bigr]\bigl( \laFb+\laFa+i \rFa_v 
\bigr)\, \phiFa{\la+i \rFa_v}{q}(\exp H)\wrt\laFa.
$$
Set $v=1/\rFa(H)$, and note that 
$\mod{\rho- \rFa_v}\geq\mod{\rFb}\geq \si$, so that 
$ \rFa_v$ is in~$\Es{\si}$. 
By the estimate (\ref{f: estimatecFbcFa}) on the Harish-Chandra function
$$
\mod{\bcFb(\laFb)}^{-2}
\leq C\, (1+\mod{\laFb})^{\sum_{\al \in \SppFb} d_{\al}} 
\leq C\, (1+\mod{\laFb+\laFa})^{\sum_{\al \in \Spp} d_{\al}} 
=    C\, (1+\mod{\la})^{n-\ell}.
$$
By Theorem \ref{t: asymptoticexpansion}~(iii), (\ref{f: estimatecFbcFainsieme}) 
and (\ref{f: estimatesinconeF}) we have that
for all $H$ in $\pareteF{F}{\de_F}{\vep_F}$
\begin{equation}
\begin{aligned} \label{f: pointwise hq}
\bigmod{h_q^F(H)}
&   \leq C\,\e^{-\rFa(H) + \vep_F\mod{\HFb}- \rFa_v(H)-\rFb(H)-q(H)}
     \int_{\fra^*} \bigmod{m\bigl(\la+i \rFa_v)}\,\wrt\mu^{n-\ell+d}(\la) \\
&   \leq C\, \e^{-2\rho(H)+(\vep_F+\mod{\rFb})\mod{\HFb}
     -\vep_F\mod{q}\mod{H}}\,\Ss{n-\ell+d}{0}(m)\\ 
&  \leq C\, \e^{-2\rho(H)+(\vep_F+\mod{\rFb})\gamma\de_F\mod{H}
     -\vep_F\height{q}\mod{H}/2}\,\e^{-\vep_F\height{q}/2} \, 
     \Ss{n-\ell+d}{0}(m).
\end{aligned}
\end{equation}
Thus, if $\de_F \leq \gamma^{-1}\,(\vep_F+\mod{\rFb})^{-1}\,\vep_F/2$, then
$$
\norm{\cM_{2\rho}h_q^F}{\ly{\Strip_2}}
\leq C\, \e^{-\vep_F\height{q}/2} \, \Ss{n-\ell+d}{0}(m). 
$$
By summing over all $q$ in $\QFa$ such that $0<\mod{q}<N$, we see that
\begin{equation}\label{f: hq1}
\Bignorm{\cM_{2\rho} \Bigl(\sum_{ q\in\QFa,\,0<\height{q}< N} h_q^F
\Bigr)}{\ly{\Strip_2}}
\leq C\, \Ss{n-\ell+d}{0}(m).
\end{equation}

Finally, we consider $h_0^F$.  
By arguing much as above, we move the contour 
of integration to the space $\fra^*+i \rFa_v$ with $v=1/\rFa(H)$. 
Then
(\ref{f: hq}) and the formula
for $\vp_{\la,0}^F$ given in Theorem \ref{t: asymptoticexpansion}~\rmi\ 
imply that for all $H$ in $\pareteF{F}{\de_F}{\vep_F}$ 
\begin{equation} \label{f: espr h0F}
h_0^F(H)
=  \mod{W_F\backslash W}
    \e^{(v-2)\rFa(H)}\,\int_{\aFb^*}
\vp_{\laFb}(\exp \HFb) \, m_0(\laFb; H^F) \,\mod{\bcFb(\laFb)}^{-2}\wrt\laFb
\end{equation}
where 
\begin{equation} \label{f: def m0}
m_0(\laFb; H^F)
= \int_{\astFa} \bigl[(\vbcFa)^{-1} m\bigr](\laFb+\laFa+i \rFa_v)\,
\e^{i\laFa(H^F)} \wrt\laFa.
\end{equation}
Define $\si_F$  by
$$
\si_F
= \min\{\mod{\rho_{F'}}: \emptyset\subset F'\subseteq F \}.
$$
Clearly $\si_F\geq \si$.  Denote by $\EsFb{\si_F}$ the
$W_F$ invariant subset of $\bW_F$ defined by
$$
\EsFb{\si_F}
= \{\eta_F\in \bW_F: \,  \mod{\eta_F- w\cdot\rFb} \geq \si_F \, \, \, \, 
\hbox{for all $w\in W_F$} \}.
$$
Observe that if $\eta_F$ is in $\EsFb{\si_F}$, 
then $\eta=\eta_F+ \rFa_v$ is in $\Es{\si}$.  Indeed,
\begin{equation} \label{f: verifica}
\bigmod{\eta_F +  \rFa_v - \rho}
=     \bigmod{\eta_F - \rFb  -v \, \rFa}
\geq  \bigmod{\eta_F - \rFb}
\geq \si_F \geq \si.
\end{equation}
Now we prove that $m_0(\cdot; H^F)$ is in
$\Sspace{\EsFb{\si_F},0}$, uniformly with respect to $H^F$.   
Indeed, for any $r$ in $[0,\infty)$
$$
S_{\EsFb{\si_F},0}^{r}(m_0(\cdot;H^F))
= \sup_{\eta_F \in \EsFb{\si_F}} \int_{(\fra^*)_F}
      \mod{m_0(\laFb+i\eta_F;H^F)} \Upsilon^r(\laFb) \wrt \la_F.
$$
By (\ref{f: estimatecFbcFa})
$$
\mod{m_0(\laFb+i\eta_F;H^F)}
\leq C\, \int_{\astFa}  \mod{m(\laFb+\laFa+i\eta_F+i \rFa_v)} \, 
\Upsilon^{n-\ell} (\laFb+\laFa) \wrt\laFa.
$$
Hence, by Tonelli's Theorem and the fact that
$
\Upsilon^r(\laFb) \, \Upsilon^{n-\ell} (\laFb+\laFa) 
\leq \Upsilon^{n-\ell+r} (\laFb+\laFa)
$
$$
\begin{aligned}
S_{\EsFb{\si_F},0}^{r}(m_0(\cdot;H^F))
& \leq \sup_{\eta_F \in \EsFb{\si_F}} \int_{\fra^*}
      \mod{m(\la +i\eta_F+i \rFa_v)} \wrt \mu^{n-\ell+r}(\la) \\
& \leq C \, S_{\Esi,0}^{n-\ell+r}(m),
\end{aligned}
$$
where $C$ is independent of $H^F$.

Note that the restriction of $\vp_{\laFb}$ to $\exp(\fra_F)$ 
may be interpreted as the restriction to $\exp(\fra_F)$ of 
an elementary spherical function
on an appropriate symmetric space of the noncompact type and rank $\mod{F}$. 
By Lemma~\ref{l: insiemi}~\rmii\ if $H$ is in $\pareteF{F}{\de_F}{\vep_F}$,
then $H_F$ is in $(\Strip_F)_2$.  
By induction, there exists $s$ in $[0,\infty)$ such that
$$
\sup_{H_F \in (\Strip_F)_2} \Bigmod{\e^{2\rho_F(H_F)} \int_{\fra^*_F} 
\vp_{\laFb}(\exp \HFb) \, m_0(\laFb; H^F) \,\mod{\bcFb(\laFb)}^{-2}\wrt\laFb}
\leq C \, S_{\EsFb{\si_F},0}^{s}\bigl(m_0 (\cdot;H^F)\bigr).
$$
Hence
\begin{equation} \label{f: h0} 
\begin{aligned}
\bigmod{h_0^F(H)}
& \leq C\,\e^{-2\rFa(H)}\,\e^{-2\rFb(H)}\,
     S_{\EsFb{\si_F},0}^{s}\bigl(m_0 (\cdot;H^F)\bigr) \\
& \leq C\,\e^{-2\rho(H)}\, S_{\Esi,0}^{n-\ell+s}(m)
   \quant H \in \pareteF{F}{\de_F}{\vep_F}.
\end{aligned}
\end{equation}
Fr{}om (\ref{f: norm rN}), (\ref{f: hq1}) and (\ref{f: h0}) we deduce that 
$$
\bignorm{\cM_{2\rho} \,\mathcal H^{-1}m }{\ly{\Strip_2}}
\leq C\,\Ss{s'}{0}(m)
\quant m\in \Sspace{\Esi,0},
$$
where $s'=\max\{n-\ell+d,n-\ell+s   \}$, and \rmi\ is proved. 

Now we prove \rmii.   
Suppose that $m$ is in $\Sspace{\Esi,J}$ with $J \geq \ell+1$.
By arguing as in the proof of \rmi, we may write
$$
k(\exp H)
= \sum_{q\in \QFa,\,\height{q}<N} h_q^F (H) + r_N^F  (H)
\quant H \in \Strip_2.
$$
Observe that if $N>\mod{\rho}/\vep_F$, 
fr{}om the pointwise estimate (\ref{f: pointwiserN}) we deduce
that
\begin{equation} \label{f: rNL1}
\begin{aligned}
\int_{\pareteF{F}{\de_F}{\vep_F}}\bigmod{r^F_N(H)}\, \e^{2\rho(H)} \wrt H
&\leq   C\,\Ss{d}{0}(m)  \, \int_{\fra^+} 
      \e^{\mod{\rho}\mod{H}-N\vep_F\mod{H}}\, \bigl( 1+\mod{H} \bigr)^{d}\wrt H\\
&\leq C\,\Ss{d}{0}(m).
\end{aligned}
\end{equation}
Similarly, if $\de_F<\ga^{-1}\,(\vep_F+\mod{\rFb})^{-1}\,\vep_F/2$,
then the pointwise estimate (\ref{f: pointwise hq}) implies that
$$
\begin{aligned}
\int_{\pareteF{F}{\de_F}{\vep_F}}\bigmod{h_q^F(H)} \, \e^{2\rho(H)} \wrt H 
&  \leq C\, \e^{-\vep_F\height{q}/2} \, \Ss{n-\ell+d}{0}(m)  
    \int_{\fra^+} \e^{(\vep_F+\mod{\rFb})\ga\de_F\mod{H}}\,
     \e^{-\vep_F\height{q}\mod{H}/2}\, \wrt H \\
&  \leq C\, \e^{-\vep_F\height{q}/2} \, \Ss{n-\ell+d}{0}(m). 
\end{aligned}
$$
By summing over all $q$ in $\QFa$ such that $0<\mod{q}<N$, we see that
\begin{equation}\label{f: hq1L1}
\int_{\pareteF{F}{\de_F}{\vep_F}}
\Bigmod{ \sum_{ q\in\QFa,\,0<\height{q}< N}  h_q^F(H) }\, \e^{2\rho(H)} \wrt H 
\leq C\, \Ss{n-\ell+d}{0}(m).
\end{equation}
It remains to estimate $\int_{\Strip_2} \mod{h_0^F(H)} \, \e^{2\rho(H)} \wrt H$.
By arguing as in the proof of \rmi, we may write   
$$
h_0^F(H)
=  \mod{W_F\backslash W}
   \, \e^{(v-2)\rFa(H)}\,\int_{\aFb^*}
\vp_{\laFb}(\exp \HFb) \, m_0(\laFb; H^F) \,\mod{\bcFb(\laFb)}^{-2}\wrt\laFb,
$$
where $m_0$ is defined in (\ref{f: def m0}).
By integrating by parts $\ell+1$ times with respect to the variable 
$\laFa$ in the integral in (\ref{f: def m0}), we see that
$$
m_0(\laFb; H)
=  \smallfrac{1}{\bigl[i\, B(H_{\rho^F},\HFa)\bigr]^{\ell+1}}\,
\, m_{\ell+1}(\laFb; H^F),
$$
where
$$
m_{\ell+1}(\laFb; H^F)
=  \int_{\astFa} \partial^{\ell+1}_{\rFa}
      \bigl[(\vbcFa)^{-1} m\bigr](\laFb+\laFa+i \rFa_v)\,
      \e^{i\laFa(H^F)}   \wrt\laFa.
$$
We claim that $m_{\ell+1}(\cdot; H^F)$ is in
$\Sspace{\EsFb{\si_F},0}$, uniformly with respect to $H^F$.   

Indeed, by Leibniz's r\^ule $m_{\ell+1}$ 
may be written as a linear combination of terms of the form
$$
\int_{\astFa} \bigl[\partial^{\ell+1-j}_{\rFa} \bigl((\vbcFa)^{-1}\bigr) 
\, \bigl(\partial^{j}_{\rFa}  m\bigr)\bigr](\laFb+\laFa+i \rFa_v)\,
      \e^{i\laFa(H)}   \wrt\laFa.
$$
where $0\leq j \leq \ell+1$.
Therefore (\ref{f: estimatecFbcFa}) implies that 
for any $\eta_F$ in $\EsFb{\si_F}$
$$
\mod{m_{\ell+1}(\laFb+i\eta_F;H^F)}
\leq C\, \sum_{j=0}^{\ell+1}
\int_{\astFa}  \mod{\partial^{j}_{\rFa}m(\laFb+\laFa+i\eta_F+i \rFa_v)} \, 
\Upsilon^{n-\ell} (\laFb+\laFa) \wrt\laFa.
$$
Hence, for any $r$ in $[0,\infty)$
$$
\begin{aligned}
S_{\EsFb{\si_F},0}^{r}\bigl(m_{\ell+1}(\cdot;H^F)\bigr)
& = \sup_{\eta_F \in \EsFb{\si_F}} \int_{(\fra^*)_F}
      \mod{m_{\ell+1}(\laFb+i\eta_F;H^F)} \Upsilon^r(\laFb) \wrt \la_F \\
& \leq C\, \sum_{j=0}^{\ell+1} \sup_{\eta_F \in \EsFb{\si_F}} \int_{\fra^*}
      \mod{\partial^{j}_{\rFa} 
      m(\la +i\eta_F+i \rFa_v)} \wrt \mu^{n-\ell+r}(\la) \\
& \leq C \, S_{\Esi,\ell+1}^{n-\ell+r}(m),
\end{aligned}
$$
thereby proving the claim.  In the last inequality we have
used the fact proved above (see (\ref{f: verifica}))
that if $\eta_F$ is in $\EsFb{\si_F}$, 
then $\eta_F+ \rFa_v$ is in $\bE_\si$. 

By \rmi\ there exists $s$ in $[0,\infty)$ such that
for all $H_F$ in $(\Strip_F)_2$
$$
 \Bigmod{\e^{2\rho_F(H_F)} \int_{\fra^*_F} 
\vp_{\laFb}(\exp \HFb) \, m_{\ell+1}(\laFb; H^F) 
\,\mod{\bcFb(\laFb)}^{-2}\wrt\laFb}
\leq C \, S_{\EsFb{\si_F},0}^{s}\bigl(m_{\ell+1} (\cdot;H^F)\bigr).
$$
Hence
$$
\begin{aligned}
\bigmod{h_0^F(H)}
& \leq C\, \smallfrac{\e^{-2\rFa(H)-2\rFb(H)}}{\mod{H^F}^{\ell+1}}\,
     S_{\EsFb{\si_F},0}^{s}\bigl(m_{\ell+1} (\cdot;H^F)\bigr) \\
& \leq C\,\smallfrac{\e^{-2\rho(H)}}{\mod{H^F}^{\ell+1}}
     \, S_{\Esi,\ell+1}^{n-\ell+s}(m).
\end{aligned}
$$
Observe that, by (\ref{f: estimatesinconeF}), 
$$
\mod{H}^2 
= \mod{H_F}^2+\mod{H^F}^2 
\leq \ga^2 \, \de_F^2 \, \mod{H}^2 + \mod{H^F}^2
\quant H \in\pareteF{F}{\de_F}{\vep_F}\cap \NBall{1}^c.
$$
Hence, if $\de_F < 1/\ga$, then
$$
\mod{H^F}^2 \geq (1-\ga^2\, \de_F^2) \, \mod{H}^2  
\quant H \in \pareteF{F}{\de_F}{\vep_F}\cap \NBall{1}^c.
$$
Therefore
$$
\begin{aligned}
\int_{\pareteF{F}{\de_F}{\vep_F}\cap \NBall{1}^c}
\bigmod{ h_0^F(H) }\, \e^{2\rho(H)} \wrt H 
&  \leq C \, \Ss{n-\ell+s}{\ell+1}(m)\, 
    \int_{\pareteF{F}{\de_F}{\vep_F}\cap \NBall{1}^c}
    \bigmod{H}^{-(\ell+1)} \wrt H \\
&  \leq C\, \Ss{n-\ell+s}{\ell+1}(m) 
    \quant m\in \Sspace{\Esi,J}.
\end{aligned}
$$
This, (\ref{f: rNL1}), (\ref{f: hq1L1}) and (\ref{f: near the origin})
imply that 
$$
\bignorm{\cM_{2\rho} \,\mathcal H^{-1}m }{\lu{\Strip_2}}
\leq C\,\Ss{s'}{\ell+1}(m)
\quant m\in \Sspace{\Esi,J},
$$
where $s'=\max\{n-\ell+d,n-\ell+s   \}$. 

This concludes the proof of \rmii\ and of the lemma. 
\end{proof}

\section{Proof of the main result} \label{s: Proof}

In the proof of Theorem~\ref{t: main result}
we use Harish-Chandra's expansion of
spherical functions away from the walls of the Weyl chamber.
Denote by $\La$ the positive lattice generated
by the simple roots in $\Sigma^+$.
For all $H$ in $\fra^+$ and $\la$ in $\fra^*$
\begin{equation} \label{f: HC expansion}
\mod{{\bf{c}}(\la)}^{-2}\,\vp_{\la}(\exp H)
= \e^{-\rho(H)}\,\sum_{q\in \La} \e^{-q(H)} \,
\sum_{w\in W} {\bf{c}}(-w\cdot\la)^{-1} \,
\Ga_q(w\cdot \la)\,\e^{i(w\cdot \la)(H)}.
\end{equation}
The coefficient $\Ga_0$ is equal to $1$; the other coefficients $\Ga_q$ are
rational functions, holomorphic in
$\TWpiut$ for some $t$ in $\BR^-$
(see (\ref{f: ngbhd of E+}) for the definition of $\TWpiut$).
Moreover, there exists a constant $d$, and,
for each positive integer $N$, another constant $C$ such that
\begin{equation} \label{f: estimates for Gaq}
\mod{D^I \Ga_q(\zeta)}
\leq C\, (1+\mod{q})^d
\quant \zeta \in \TWpiut
\quant I: \mod{I} \leq N.
\end{equation}
Note that the estimate for the derivatives
is a consequence of Gangolli's estimate for~$\Ga_q$ \cite{Ga}
and Cauchy's integral formula.  The Harish-Chandra 
expansion is pointwise convergent in $\fra^+$ and uniformly
convergent in $\fra^+ \setminus \Strip_c$ for every $c>0$.

\begin{remark} \label{rem: c inversa}
Suppose that $L$ is a positive integer.
There exists a constant $C$ such that 
$$
\bignorm{(\hbox{\vbc})^{-1}\, \Ga_q\, m\, \bigl[1-(1-\e^{-Q})^L\bigr] 
\, \e^{Q/2}}{H'(\TWpiut;J,\kappa)}
\leq C \, (1+\mod{q})^d \, \norm{m}{H'(\TW;J,\kappa)}
$$
for all $m$ in $H'(\TW;J,\kappa)$ and for all $q$ in $\La$.
Similarly, 
there exists a constant $C$ such that 
$$
\bignorm{(\hbox{\vbc})^{-1}\, \Ga_q\, (M\circ Q)\, 
\bigl[1-(1-\e^{-Q})^L\bigr] 
\, \e^{Q/2}}{H(\TBpiut;J,\kappa)}
\leq C \, (1+\mod{q})^d \, \norm{M}{\fH(\bP;J,\kappa)}
$$
for all $M$ in $\fH(\bP;J,\kappa)$ and for all $q$ in $\La$.

To prove the first estimate we 
compute derivatives of order at most $J$ 
of $(\hbox{\vbc})^{-1}\, \Ga_q\, m_{B}\, \bigl[1-(1-\e^{-Q})^L\bigr] 
\, \e^{Q/2}$ by using Leibnitz's r\^ule.
To estimate each of the summands,
we use (\ref{f: estimates for Gaq}),
and the fact that for some $t$ in $\BR^-$
the function $(\hbox{\vbc})^{-1}$ is holomorphic in
$\TWpiut$, and both
$(\hbox{\vbc})^{-1}$ and its derivatives
grow at most polynomially at infinity in $\TWpiut$ (see 
(\ref{f: derivata c check})).

The proof of the second estimate is similar and is omitted.
\end{remark}

\begin{remark} \label{rem: integrale ion}
Observe that if $\kappa<1$, then for every $c$ in $\BR^+$
$$
\int_{\fra^+\setminus \Strip_c} \frac{\e^{-\om(H)}}{\bigl[1
+\cN(H)\bigr]^{\ell+1-2\kappa}} \wrt H
< \infty
\qquad\hbox{and} 
\qquad
\int_{\fra^+\setminus \Strip_c} \frac{\e^{\rho(H)-\mod{\rho}\, \mod{H}
-\om(H)}}{\bigl[1 +\rho(H)\bigr]^{(\ell-1)/2}} \wrt H
< \infty.
$$
We prove that the first integral above is convergent.  
The proof that the second is convergent is easier, and is omitted.

Observe that there exists $\vep$ in $\BR^+$
such that $\om(H) \geq \vep \, \mod{H}$ for all $H$ in
$\Cone{c_0} \setminus \Strip_c$. Therefore
$$
%\begin{aligned}
\int_{\Cone{c_0}\setminus \Strip_c} \frac{\e^{-\om(H)}}{\bigl[1
+\cN(H)\bigr]^{\ell+1-2\kappa}} \wrt H
 \leq \int_{\Cone{c_0}\setminus \Strip_c} \e^{-\vep \, \mod{H}} \wrt H 
 < \infty.
%\end{aligned}
$$
Moreover, there exists a constant $C$ such that 
$
\cN(H) \geq C\, \mod{H'} \geq C\, \rho(H)
$
for every $H$ in $\fra^+\setminus (\Strip_c\cup \Cone{c_0})$.  Hence
$$
\begin{aligned}
\int_{\fra^+\setminus (\Strip_c\cup \Cone{c_0})} \frac{\e^{-\om(H)}}{\bigl[1
+\cN(H)\bigr]^{\ell+1-2\kappa}} \wrt H
& \leq C \, \int_{\fra^+} \frac{\e^{-\om(H)}}{\bigl[1
      +\rho(H)\bigr]^{\ell+1-2\kappa}} \wrt H,
\end{aligned}
$$
which is easily seen to be convergent \cite[Lemma 3.5]{I3}.
\end{remark}

Now we prove our main result, which we
restate for the reader's convenience.

\begin{theorem*}[\textbf{2.10}]
Denote by $J$ the integer $[\![n/2]\!]+1$.
The following hold:
\begin{enumerate}
\item[\itemno1]
if $\kappa$ is in $[0,1)$, then there exists a constant $C$ such that
for all $B$ in $\GOp{2}$ for which $m_B$ is in $H'(\TW;J,\kappa)$
$$
\opnorm{B}{1;1,\infty}
\leq C \, \norm{m_B}{H'(\TW;J,\kappa)};
$$
\item[\itemno2]
there exists a constant $C$ such that 
$$
\opnorm{M(\cL)}{1;1,\infty}
\leq C \, \norm{M}{\fH(\bP;J,1)}
\quant M \in \fH(\bP;J,1).
$$
\end{enumerate}
\end{theorem*}

\begin{proof}
First we prove \rmi.
Suppose that $L$ is a positive integer $> \kappa + J$. 
We denote by $B_1$ and~$B_2$ the operators defined by
$$
B_1 = B \, \bigl(1-\e^{-\cL} \bigr)^{L}
\qquad\hbox{and}\qquad
B_2 = B \, \bigl[1-\bigl(1-\e^{-\cL} \bigr)^{L}\bigr].
$$
Thus,
$
B= B_1 + B_2.
$
Denote by $h_1$ the heat kernel at time $1$ (see (\ref{f: heat kernel})). 
The spherical multipliers associated to $B_1$ and $B_2$ are the functions
$m_{B_1}$ and $m_{B_2}$ on $\TW$ defined by
$$
m_{B_1} = m_B \, \bigl(1-\wt h_1\bigr)^{L}
\qquad\hbox{and}\qquad
m_{B_2} = m_B \, \bigl[1-\bigl(1-\wt h_1\bigr)^{L}\bigr].
$$
Denote by $\psi$ a smooth $K$--bi-invariant
function such that $\psi(\exp H) = 0$ for $H$ in $\Strip_1\cap \NBall{2}^c$,
and $\psi(\exp H) = 1$ for $H$ in $\Strip_2^c \cup \NBall{1}$.
We decompose~$\kBdue$ as follows
$$
\kBdue
= (1-\psi)\,\kBdue + \psi \, \kBdue.
$$

\emph{Step I: $B_1$ is of weak type $1$}.
Since ${L}>\kappa + J$,
the function $m_{B_1}$ and its derivatives
up to the order $J$ are bounded on $\TW$.  This is due to the fact
that $(1-\wt h_1)^{L}$ vanishes at the point $i\rho$,
together with all its derivatives up to the order $L-1$,
and this compensates for the fact that $m_{B_1}$ and
its derivatives may be unbounded near $i\rho$.
A straightforward computation shows that $m_{B_1}$ satisfies
the hypotheses of 
\cite[Corollary 17]{A2}. Therefore~$B_1$ is of weak type~$1$,
and $\opnorm{B_1}{1;1,\infty} \leq C \, \norm{m}{H'(\TW;J,\kappa)}$.

\emph{Step II: estimates away from the walls}.
We claim that the function $\psi\, \kBdue$
may be written as the sum of two
$K$--bi-invariant functions $\kBdueO$ and $\kBdueI$,
where $\kBdueI$ is in $\lu{\KGK}$ and $\kBdueO$
satisfies the following estimates in Cartan co-ordinates:
there exists a constant $C$ such that for all
$H$ in $\fra^+ \setminus \Strip_1$
\begin{equation} \label{f: kB2 Cartan i}
\bigmod{\kBdueO (\exp H)}
\leq 
\begin{cases}
C\, \norm{m_B}{H'(\TW;J,\kappa)}\,
\e^{-2\rho(H)}\, \log\bigl(2+ \cN(H)\bigr) \,\bigl[1+ \cN(H)\bigr]^{-\ell-1}
& \hbox{if $\kappa =0$} \\
C\, \norm{m_B}{H'(\TW;J,\kappa)}\,
\e^{-2\rho(H)} \, \bigl[1+ \cN(H)\bigr]^{2\kappa-\ell-1}
& \hbox{if $0< \kappa \leq 1$} \\
\end{cases}
\end{equation}
(see (\ref{f: anisotropic norm}) 
for the definition of $\cN$).

To prove this, 
we observe preliminarily
that if $H$ is in $\fra^+ \setminus \Strip_1$ and
$q = \sum_{\al \in \Sppp} n_{\al} \, \al$, then
\begin{equation} \label{f: serie HC 0}
q(H)
= \sum_{\al \in \Sppp} n_{\al} \, \al(H)
  \geq \om(H) \sum_{\al \in \Sppp} n_{\al}
= \om(H) \mod{q}
%\geq \mod{q},
\end{equation}
so that
\begin{equation} \label{f: serie HC 1}
%\sup_{H\in \fra^+\setminus\Strip_1} \,
\sum_{q\in \La\setminus \{0\}} \e^{-q(H)}\, (1+\mod{q})^d
\leq  \e^{-\om(H)} \sum_{q\in \La\setminus \{0\}} \e^{1-\mod{q}}\, (1+\mod{q})^d
\leq C\,\e^{-\om(H)}.  
\end{equation}
This, (\ref{f: estimates for Gaq}) and (\ref{f: derivata c check})
(with $I = 0$) imply that 
\begin{equation} \label{f: serie HC 2}
\begin{aligned}
\sum_{q\in \La\setminus \{0\}} \e^{-q(H)} 
\int_{\fra^*} \mod{m_{B_2}(\la) \,{\bf{c}}(-\la)^{-1}\,
\Ga_q( \la)\,\e^{i\la(H)}} \wrt\la
& \leq C \, \norm{m_B}{\ly{\fra^*}} 
    \, \int_{\fra^*} \e^{-Q(\la)/2} \wrt\la \\
& \leq C\, \norm{m_B}{H'(\TW;J, \kappa)}. 
\end{aligned}
\end{equation}

Now, we substitute Harish-Chandra expansion (\ref{f: HC expansion})
in the inversion formula 
$$
k_B (\exp H)
= c_G \,\int_{\fra^*} m_B (\la)\,\vp_{\la}(\exp H)\,\wrt\mu(\la)
\quant H \in\fra^+,
$$
use the fact that the integrand is Weyl invariant, and obtain
$$
\psi(H)\, \kBdue(\exp H)
= c_G\, \mod{W} \,\psi(H)\,\e^{-\rho(H)}\,\sum_{q\in \La}\e^{-q(H)}\, 
\int_{\fra^*} m_{B_2}(\la) \,{\bf{c}}(-\la)^{-1}\,
\Ga_q( \la)\,\e^{i\la(H)} \wrt\la,
$$
where $\mod{W}$ denotes the cardinality of the Weyl group,
and the term by term integration is justified by (\ref{f: serie HC 2}).
Write $\psi\, \kBdue = \kBdueO+\kBdueI$, where
$$
\begin{aligned}
\kBdueO(\exp H)
& = c_G\, \mod{W} \,\psi(H)\,\e^{-\rho(H)}\,
  \int_{\fra^*} m_{B_2}(\la) \,{\bf{c}}(-\la)^{-1}\, \e^{i\la(H)} \wrt\la \\
\kBdueI(\exp H)
& = c_G\, \mod{W} \,\psi(H)\,\e^{-\rho(H)}\,
\sum_{q\in \La\setminus\{0\}}\e^{-q(H)}\,
\int_{\fra^*} m_{B_2}(\la) \,{\bf{c}}(-\la)^{-1}\,
\Ga_q( \la)\,\e^{i\la(H)} \wrt\la.
\end{aligned}
$$

To prove estimate (\ref{f: kB2 Cartan i}) for $\kBdueO$
in the case where $0< \kappa \leq 1$, 
we apply Lemma~\ref{l: pointwiseestimates}~\rmi\
(with $(\hbox{\vbc})^{-1}\, m_{B_2}\, \e^{Q/2}$ in place of~$m$),
and then Remark~\ref{rem: c inversa} (with $q=0$), and obtain that
$$
\begin{aligned}
\bigmod{\kBdueO(\exp H)}
& \leq C \, \bignorm{(\hbox{\vbc})^{-1}\, m_{B_2}\, 
   \e^{Q/2}}{H'(\TWpiu;J,\kappa)}
   \, \frac{\e^{-2\rho(H)}}{\bigl[1+\cN(H)\bigr]^{\ell+1-2\kappa}} \\
& \leq C \, \norm{m_{B}}{H'(\TW;J,\kappa)}
   \, \frac{\e^{-2\rho(H)}}{\bigl[1+\cN(H)\bigr]^{\ell+1-2\kappa}} 
\quant H \in \fra^+,
\end{aligned}
$$
as required.  The required estimate for $\kappa = 0$ is proved similarly.

It remains to show that $\kBdueI$ is in $\lu{\KGK}$ for all $\kappa$
in $[0,1]$.  We give the details when $0 < \kappa \leq 1$.  Those in 
the case where $\kappa =0$ are similar, and are omitted.
We apply Lemma~\ref{l: pointwiseestimates}~\rmi\
(with the function 
$(\hbox{\vbc})^{-1}\,\Ga_q\, m_{B_2}\, \e^{Q/2}$ in place of $m$)
to each summand of the series that appears
in the definition of $\kBdueI$,
and obtain that 
$$
\begin{aligned}
\bigmod{\kBdueI(\exp H)}
& \leq C \, \frac{\e^{-2\rho(H)}}{\bigl[1+\cN(H)\bigr]^{\ell+1-2\kappa}}
       \sum_{q\in \La\setminus\{0\}} \e^{-q(H)}\,
       \bignorm{(\hbox{\vbc})^{-1}\,\Ga_q\,
       m_{B_2}\, \e^{Q/2}}{H'(\TWpiu,J,\kappa)} \\
& \leq C\, \norm{m_B}{H'(\TW;J,\kappa)}
    \, \frac{\e^{-2\rho(H)-\om(H)}}{\bigl[1+\cN(H)\bigr]^{\ell+1-2\kappa}}
\quant H \in \fra^+\setminus \Strip_1,
\end{aligned}
$$
where we have used Remark~\ref{rem: c inversa},
(\ref{f: serie HC 0}) and (\ref{f: serie HC 1}).
Therefore
$$
\begin{aligned}
\norm{\kBdueI}{\lu{G}}
& \leq C \, \norm{m_B}{H'(\TW;J,\kappa)} \,
    \int_{\fra^+\setminus \Strip_1} \frac{\e^{-\om(H)}}{\bigl[1
    +\cN(H)\bigr]^{\ell+1-2\kappa}} \wrt H \\
& \leq C \, \norm{m_B}{H'(\TW;J,\kappa)}.
\end{aligned}
$$
where we have used Remark~(\ref{rem: integrale ion}).
This concludes the proof of Step II.

\emph{Step III: estimates near the walls}.
We shall prove that the function $(1-\psi)\, \kBdue$ is integrable. 
By Lemma~\ref{l: inductivestep}~\rmii\
there exists an integer $s$ such that
$$
\begin{aligned}
\norm{(1-\psi)\, k_{B_2}}{\lu{X}}
& \leq C \, \norm{\cM_{2\rho}k_{B_2}}{\lu{\Strip_2}} \\
& \leq C \, \Ss{s}{\ell+1}(m_{B_2}).
\end{aligned}
$$
To conclude the proof of Step III it suffices
to show that there exists a constant $C$ 
such that 
\begin{equation} \label{f: conclusione I}
\Ss{s}{\ell+1}(m_{B_2})
\leq C\, \,\norm{m_B}{H'(\TW;J,\kappa)}.
\end{equation}
Indeed, by Leibniz's r\^ule there exists a constant $C$ such that 
for every multiindex $I$ with $\mod{I}\leq \ell+1$ 
and for every $\zeta$ in $\TWpiu$
$$
\mod{D^I m_{B_2}(\zeta)}
\leq C\,\, \norm{m_B}{H'(\TW;J,\kappa)}\, \e^{-\Re Q(\zeta)/2}\, 
\max\bigl[ \mod{Q(\zeta)}^{-|I|/2},
\mod{Q(\zeta)}^{-\kappa-|I|'/2} \bigr].
$$
Then for every $\eta$ in $\Es{\si}$ 
and for every multiindex $I$ with $\mod{I}\leq \ell+1$ 
$$
\begin{aligned}
& \int_{\fra^*}\,\mod{D^I m_{B_2}(\la+i\eta)}\wrt\mu^s(\la)   \\
%& \leq \int_{\fra^*\setminus \NBall{1} }(1+\mod{\la+i\eta})^s\,
%    \mod{m_{B_2}(\la+i\eta)}\wrt\la+ 
%    C\,\int_{\NBall{1}}\mod{m_{B_2}(\la+i\eta)}\wrt\la\\
& \leq C\,\, \norm{m_B}{H'(\TW;J,\kappa)}\,
    \Bigl[\int_{\fra^*\setminus \Ball{R}}
    (1+\mod{\la})^{s-\mod{I}/2} \,\e^{-\Re Q(\la+i\eta)/2} \wrt\la
    +\int_{\Ball{R}}\mod{Q(\la+i\eta)}^{-\kappa-\mod{I}'/2}\wrt\la\Bigr],
\end{aligned}
$$
where $R$ is large enough. 
Observe that the first integral on the right hand side is dominated by 
$C\, \int_{\fra^+} \e^{-\mod{\la}^2/3} \wrt \la$, where
$C$ is a constant depending on $s$, but not on $\eta$.
Furthermore, since $\mod{Q(\la + i\eta)}$ is continuous and
does not vanish 
when $\eta$ is in $\Es{\si}$ and $\la$ stays in a compact neighbourhood
of the origin,
we may conclude that it is bounded away from $0$.  
Thus, the second integral on the right hand
side in the formula above is finite, and (\ref{f: conclusione I}) 
is proved.

\emph{Step IV: conclusion}.
Recall that
$$
\kBdue = \kBdueO + \kBdueI + (1-\psi)\, \kBdue,
$$
and that $\kBdueI$ and $(1-\psi)\, \kBdue$ are in $\lu{\KGK}$.
Thus, the operators $f\mapsto f*\kBdueI$ and
$f\mapsto f*\bigl[(1-\psi)\, \kBdue\bigr]$ are bounded on $\lu{X}$, hence,
\emph{a fortiori}, of weak type $1$.
The estimates proved in Step~II imply that
the convolution operator $f\mapsto f*\kBdueO$ is
of weak type~$1$ by Proposition~\ref{p: debole per tau}.
Therefore $B_2$ is of weak type~$1$. Since $B_1$ is of weak type
$1$ (see Step~I), we may conclude that $B$ is of weak type~$1$,
as required to conclude the proof of \rmi.

The proof of \rmii\ is similar to the proof of \rmi.  We briefly
indicate the changes needed.  
We decompose $M(\cL)$ as the sum $M_1(\cL) + M_2(\cL)$, where 
$M_1$ and $M_2$ are the functions defined~by 
$$
M_1 (z) = M(z) \, \bigl(1-\e^{-z} \bigr)^{L}
\qquad\hbox{and}\qquad
M_2 (z) = M(z) \, \bigl[1-\bigl(1-\e^{-z} \bigr)^{L}\bigr].
$$
We denote by $m_{M_1(\cL)}$ and $m_{M_2(\cL)}$ the 
spherical multipliers associated to $M_1(\cL)$ 
and to $M_2(\cL)$ respectively. 
We write $k_{M_2(\cL)} = (1-\psi) \, k_{M_2(\cL)}+ \psi\,  k_{M_2(\cL)}$,
where $\psi$ is the defined at the beginning of
the proof of \rmi. 
By arguing as in Step~I above, we see that $M_1(\cL)$
is of weak type $1$ and that
$\opnorm{M_1(\cL)}{1;1,\infty} \leq C \, \norm{M}{\fH(\bP;J,1)}$.

We claim that the function $\psi\, \kMdue$
may be written as the sum of two
$K$--bi-invariant functions $\kMdueO$ and $\kMdueI$,
where $\kMdueI$ is in $\lu{\KGK}$ and $\kMdueO$
satisfies the following estimates in Cartan co-ordinates
\begin{equation} \label{f: kM2 Cartan ii}
\bigmod{\kMdueO (\exp H)}
\leq C\, \norm{M}{\fH(\bP;J,1)}
\, {\e^{-\rho(H)- \mod{\rho}\, \mod{H}}} \, {\bigl[1+ \rho(H)
\bigr]^{(1-\ell)/2}} \quant H \in \fra^+ \setminus \Strip_1.
\end{equation}
Indeed, since $M$ is in $\fH(\bP;J,1)$, $M\circ Q$ is in $H(\TB;J,1)$
by Proposition~\ref{p: McircQ}~\rmi.
Then we may apply Lemma~\ref{l: pointwiseestimates}~\rmii\
(with $(\hbox{\vbc})^{-1}\, (M_2\circ Q)\, \e^{Q/2}$ in place of $m$),
and obtain that
$$
\begin{aligned}
\bigmod{\kMdueO(\exp H)}
&  \leq C \, \bignorm{(\hbox{\vbc})^{-1}\, (M_2\circ Q)\,
      \e^{Q/2}}{H(\TB;J,1)} \, {\e^{-\rho(H) - \mod{\rho}\, 
      \mod{H}}} \, {\bigl[1+\rho(H)\bigr]^{(1-\ell)/2}} \\
&  \leq C \, \norm{M}{\fH(\bP;J,1)} {\e^{-\rho(H) - \mod{\rho}\, 
      \mod{H}}} \, {\bigl[1+\rho(H)\bigr]^{(\ell-1)/2}} 
\quant H \in \fra^+,
\end{aligned}
$$
thereby proving (\ref{f: kM2 Cartan ii}).  Notice that we have used
Remark~\ref{rem: c inversa} in the last inequality.

It remains to show that $\kMdueI$ is in $\lu{\KGK}$.  By arguing
as in Step~II above, we see that $\kMdueI$ satisfies 
the following estimate
$$
\bigmod{\kMdueI(\exp H)}
\leq C\, \norm{M}{\fH(\bP;J,1)}
    \, {\e^{-\rho(H)-\mod{\rho}\, \mod{H}-\om (H)}} \, {\bigl[1+\rho(H)
    \bigr]^{(1-\ell)/2}}.
$$
We now use
Remark~(\ref{rem: integrale ion}), and obtain that
$$
\begin{aligned}
\norm{\kMdueI}{\lu{G}}
& \leq C \, \norm{M}{\fH(\bP;J,1)} \,
    \int_{\fra^+\setminus \Strip_1} 
    \, \frac{\e^{\rho(H)-\mod{\rho}\, \mod{H}- \om(H)}}{\bigl[1+\rho(H)
    \bigr]^{(\ell-1)/2}} \wrt H \\
& \leq C \, \norm{M}{\fH(\bP;J,1)}.
\end{aligned}
$$

The proof that the function $(1-\psi)\, \kMdue$ is integrable
with $\norm{(1-\psi)\, \kMdue}{\lu{\KGK}}\leq C \, \norm{M}{\fH(\bP;J,1)}$,
is almost \emph{verbatim} the same as the proof of the 
corresponding statement in case \rmi\ (see Step~III),
and is omitted. 
The required conclusion follows as in Step~IV in case~\rmi.

The proof of \rmii, and of the theorem, is complete.
\end{proof}

\end{document}